%% file: quality.tex
\documentclass[a4paper,11pt,leqno]{amsart}

\usepackage{a4wide,enumerate}
\usepackage{amsmath,amssymb}
\usepackage{hyperref,delarray}
\usepackage[english]{babel}
\usepackage[english,linesnumbered]{algorithm2e}
\usepackage[utf8]{inputenc}
\usepackage{srcltx}
\usepackage{color,booktabs,float}
\usepackage{graphics}

\numberwithin{equation}{section}
\newtheorem{theo}{Theorem}[section]
\newtheorem{prop}[theo]{Proposition}
\newtheorem{coro}[theo]{Corollary}

\newtheorem{defi}[theo]{Definition}
\theoremstyle{remark}
\newtheorem*{rem*}{Remark}
\newtheorem{rem}[theo]{Remark}
\newtheorem*{acknowledgements}{Acknowledgements}
\SetKwFunction{BDyDF}{BDyDF}
\SetKwFunction{Bound}{Bound}
\SetKwFunction{NDelta}{NDelta}
\SetKwFunction{OptimalT}{OptimalT}
\SetKwFunction{floor}{floor}
\SetKwFunction{GRHcheck}{GRHcheck}
\expandafter\ifx\csname SetAlgoLined\endcsname\relax\def\SetAlgoLined{\relax}\fi

\expandafter\def\csname HyPsd@CatcodeWarning\endcsname#1{\relax}

\newcommand{\case}{&\text{if }}
\newcommand{\Chi}{{\raise0.2ex\hbox{$\chi$}}}
\newcommand{\Cl}{\mathcal C\!\ell}
\renewcommand{\d}{\,\mathrm d}
\DeclareMathOperator{\dilog}{dilog}
\newcommand{\DK}{\Delta_\K}
\newcommand{\e}{\mathbf e}
\newcommand{\g}{\gamma}
\newcommand{\good}{$\K$-good}
\newcommand{\GRH}{\textrm{\upshape GRH}}
\newcommand{\IF}{{\mathcal I}}
\DeclareMathOperator{\Iint}{I}
\DeclareMathOperator{\Imm}{Im}
\DeclareMathOperator{\Jint}{J}
\newcommand{\K}{\mathbf K}
\newcommand{\LC}{\mathcal{L}}
\newcommand{\lDK}{\log\DK}
\newcommand{\lDKsq}{\log^2\DK}
\newcommand{\llDK}{\log\lDK}
\newcommand{\N}{\mathrm N}
\newcommand{\NK}{N_\K}
\newcommand{\nK}{n_\K}
\newcommand{\NM}{\mathbf N}
\newcommand{\nt}{\notag\\}
\newcommand{\OC}{\mathcal{O}}
\newcommand{\OK}{\OC_{\K}}
\newcommand{\pG}{\mathfrak p}
\newcommand{\Phie}{\Phi_\e}
\newcommand{\QM}{\mathbf Q}
\DeclareMathOperator{\Ree}{Re}
\newcommand{\RM}{\mathbf R}
\newcommand{\SC}{\mathcal{S}}

\newcommand{\TC}{T_{\mathcal{C}}}
\newcommand{\TCK}{\TC(\K)}
\newcommand{\Te}{T_\e}
\newcommand{\TeK}{\Te(\K)}
\newcommand{\WC}{\mathcal W}
\newcommand{\zK}{{\zeta_\K}}
\newcommand{\ZM}{\mathbf Z}
\newcommand{\MK}{{M_\K}}
\newcommand{\Ell}{\mathcal{L}}
\newcommand{\Ef}{\ell}

\newcounter{fixedfig}
\newenvironment{fixedfig}{
\refstepcounter{fixedfig}
\centerline\bgroup%
\def\caption##1{\textsc{Figure \arabic{fixedfig}}: ##1}%
\begin{tabular}{c}%
}
{%
\end{tabular}%
\egroup%
\pagebreak[1]
}

\allowdisplaybreaks[4]

\title{Explicit bounds for generators of the class group}

\author[L.~Grenié]{Loïc Grenié}
\address[L.~Grenié]{Dipartimento di Ingegneria gestionale,
        dell'informazione e della produzione\\
        Università degli Studi di Bergamo\\
        viale Marconi 5\\
        24044 Dalmine (BG)
        Italy}
\email{loic.grenie@gmail.com}

\author[G.~Molteni]{Giuseppe Molteni}
\address[G.~Molteni]{Dipartimento di Matematica\\
        Università di Milano\\
        via Saldini 50\\
        20133 Milano\\
        Italy}
\email{giuseppe.molteni1@unimi.it}


\begin{document}
\makeatletter%
\let\l@icbegin=\begin%
\def\begin{\src@spec\l@icbegin}%
\let\l@icend=\end%
\def\end#1{\l@icend{#1}\src@spec\tracingmacros0}%
\makeatother%
\bibliographystyle{amsalpha}
\begin{abstract}
Assuming Generalized Riemann's Hypothesis, Bach proved that the class group $\Cl_\K$ of a number field
$\K$ may be generated using prime ideals whose norm is bounded by $12\lDKsq$, and by $(4+o(1))\lDKsq$
asymptotically, where $\DK$ is the absolute value of the discriminant of $\K$. Under the same assumption,
Belabas, Diaz y Diaz and Friedman showed a way to determine a set of prime ideals that generates $\Cl_\K$
and which performs better than Bach's bound in computations, but which is asymptotically worse. In this
paper we show that $\Cl_\K$ is generated by prime ideals whose norm is bounded by the minimum of
$4.01\lDKsq$, $4\big(1+\big(2\pi e^{\g})^{-\nK}\big)^2\lDKsq$ and $4\big(\lDK+\llDK-(\g+\log
2\pi)\nK+1+(\nK+1)\frac{\log(7\lDK)}{\lDK}\big)^2$. Moreover, we prove explicit upper bounds for the size
of the set determined by Belabas, Diaz y Diaz and Friedman's algorithms, confirming that it has size
$\asymp (\lDK\llDK)^2$. In addition, we propose a different algorithm which produces a set of generators
which satisfies the above mentioned bounds and in explicit computations turns out to be smaller than
$\lDKsq$ except for $7$ out of the $31292$ fields we tested.
\end{abstract}

\maketitle
\section{Introduction}
Let $\K$ be a number field of degree $\nK\geq 2$, with $r_1$ (resp. $r_2$) real (resp. pair of complex)
embeddings and denote $\DK$ the absolute value of its discriminant. Throughout the paper $\pG$ will
always denote a non-zero prime ideal, $\rho$ a non-trivial zero of $\zK$, $\g$ the imaginary part of such
a $\rho$ and, since we are assuming Generalized Riemann's Hypothesis, $\rho=\frac{1}{2}+i\g$. The
Euler--Mascheroni constant will also be denoted $\g$, but the context will make it clear.

Buchmann's algorithm is an efficient method to compute both the class group $\Cl_\K$ and a basis for a
maximal lattice of the unity group of $\K$. However, it needs as input a set of generators for $\Cl_\K$.
Minkowski's bound shows that ideals having a norm (essentially) below $\sqrt\DK$ may be used, but it
works just for a few fields since the discriminant increases very quickly. Assuming Generalized
Riemann's Hypothesis, Eric Bach proved in~\cite{Bach:explicit} that ideals with a norm below $12\lDKsq$
suffice, and that the bound improves up to $(4+o(1))\lDKsq$ as $\DK$ diverges, where the function in
$o(1)$ is not made explicit in that paper, but has order at least $\log^{-2/3}\DK$.
%
%
This is a remarkable improvement, but for certain applications it is still too large. A different method
to find a good bound $T$ for norms of ideal generating $\Cl_\K$ has been proposed by Karim Belabas,
Francisco Diaz y Diaz and Eduardo Friedman~\cite{small-generators}. In all tests their method behaves
very well, producing a good bound $T(\K)$ (see Section~3 of~\cite{small-generators}) which is much lower
than $4\lDKsq$. However, the authors prove \cite[Theorem~4.3]{small-generators} that $T(\K)\geq
\big(\big(\frac1{4\nK}+o(1)\big)\lDK\llDK\big)^2$, and advance the conjecture that
$T(\K)\sim\big(\frac14\lDK\llDK\big)^2$. Thus, its impressive performance is the combined effect of
relatively large/small constants in front of these bounds and of the present computational power, but
which will disappear for large $\DK$.
\smallskip

In this paper we first prove in Theorem~\ref{theo:Teasynt} an explicit, easy and better version of Bach's
$(4+o(1))\lDKsq$ bound, in Corollary~\ref{coro:2} that $4\lDKsq$ is sufficient for a wide range of fields
and in Corollary~\ref{coro:Bach4.01} that $4.01\,\lDKsq$ is sufficient for all fields.
Corollary~\ref{coro:2} also contains an explicit bound showing that the universal constant $4.01$ actually
decays exponentially to $4$ with the degree of the field.\\
Secondly, in Theorem~\ref{estim-T(K)} we prove that $T(\K)\leq \big(\frac{1+o(1)}4\lDK\llDK\big)^2$, and
that $T(\K)\leq 3.9(\lDK\llDK)^2$ with only three exceptions which are explicit. In a private
communication K.~Belabas told us that he also has a proof for the first part of this claim. Together with
the lower bound in~\cite{small-generators} it shows in particular that $T(\K)\asymp(\lDK\llDK)^2$ for
fixed $\nK$.
\smallskip

The weight function of~\cite{small-generators} can be seen as the convolution square of a characteristic
function, i.e. of a one step function. Using the convolution square of a two (resp. three) steps
function, we show in Corollary~\ref{coro:boundT1cst} that the bound already improve to
$(6.04+o(1))\lDKsq$ (resp. $(4.81+o(1))\lDKsq$),
%
%
where moreover in both cases $o(1)<0$ for fields of degree $\nK\geq 3$. To further improve the result we
propose in Subsections~\ref{subsec:algo}--\ref{subsec:algoend} a different algorithm producing a new
bound $T_1(\K)$, and which is essentially a multistep version of Belabas, Diaz y Diaz and Friedman's
algorithm, where the number of steps is not set in advance. By design, it performs better than $T(\K)$
and is lower than all the bounds we have proved in the first part of the paper.
In Subsection~\ref{subsec:tests} we report the conclusions about extensive tests we have conducted on a
few thousands of pure and biquadratic fields: in all cases the algorithm produces $T_1(\K)$ lower
than $\lDKsq$ except for some biquadratic fields where it is anyway $\leq 1.004\lDKsq$.\\
All `little-$o$' terms in these formulas are explicit, simple, of order $\frac{\llDK}{\lDK}$ and with
small coefficients.
\smallskip

We have made two sample implementations of our algorithms. The first one is a script for
PARI/GP~\cite{PARI2} which can be found at the following address:\\
\url{http://users.mat.unimi.it/users/molteni/research/generators/bounds.gp}\\
The other is the branch \texttt{loic-bnf-optim} of the \texttt{git} tree of PARI/GP, available at\\
\url{http://pari.math.u-bordeaux.fr/git/pari.git}
\smallskip

\begin{acknowledgements}
We wish to thank Giacomo Gigante for his comments and interesting discussions, and the referee for
her/his useful comments. We also thank the referee of an early version of the second part of this paper
for the suggestion to use Cholesky's decomposition. The authors are members of the INdAM group GNSAGA.
\end{acknowledgements}
\enlargethispage{\baselineskip}

\section{Preliminary}
\begin{defi}
Let $\WC$ be the set of functions $F\colon [0,{+\infty})\to\RM$ such that
\begin{itemize}
\item $F$ is continuous;
\item $\exists\varepsilon>0$ such that the function
$F(x)e^{(\frac12+\varepsilon)x}$ is integrable and of bounded variation;
\item $F(0)>0$;
\item $(F(0)-F(x))/x$ is of bounded variation.
\end{itemize}
Let then, for $T>1$, $\WC(T)$ be the subset of $\WC$ such that
\begin{itemize}
\item $F$ has support in $[0,\log T]$;
\item the Fourier cosine transform of $F$ is non-negative.
\end{itemize}
\end{defi}

\begin{defi}
For any compactly supported function $F$ on $[0,\infty)$, we set
\[
\Iint(F) := \int_{0}^{+\infty}\frac{F(0)-F(x)}{2\sinh(x/2)}\d x
\quad\text{and}\quad
\Jint(F) := \int_{0}^{+\infty}\frac{F(x)}{2\cosh(x/2)}\d x.
\]
\end{defi}

\begin{defi}
Let $\TCK$ be the lowest $T$ such that the set $\{\pG\colon\N\pG\leq T\}$ generates $\Cl_\K$.
\end{defi}

The main result of \cite[Th.~2.1]{small-generators} can be reformulated as follows.
\begin{theo}[\textbf{Belabas, Diaz y Diaz, Friedman}]\label{theoKB}
Let $\K$ be a number field satisfying the Riemann Hypothesis for all $\mathrm L$-functions attached to
non-trivial characters of its ideal class group $\Cl_\K$, and suppose that there exists, for some $T>1$,
an $F\in\WC(T)$ such that
\begin{equation}\label{theoeq}
2\sum_\pG\log\N\pG\sum_{m=1}^{+\infty}\frac{F(m\log\N\pG)}{\N\pG^{m/2}}
>
F(0)(\lDK - (\g + \log 8\pi)\nK)
+ \Iint(F) \nK
- \Jint(F) r_1.
\end{equation}
Then $\TCK<T$.
\end{theo}
Assuming \GRH, Weil's Explicit Formula (see~\cite[Ch.~XVII, Th.~3.1]{Lang:algnumtheory}), as simplified by
Poitou in~\cite{Poitou:petits-discs}, can be written for $F\in\WC$ as
\begin{multline}\label{WEF}
2\sum_\g \int_{0}^{+\infty}F(x)\cos(x\g) \d x
=
4\int_0^{+\infty}F(x)\cosh\Big(\frac{x}{2}\Big)\d x                        \\
- 2\sum_\pG\log\N\pG\sum_{m=1}^{+\infty}\frac{F(m\log\N\pG)}{\N\pG^{m/2}}
+ F(0)(\lDK - (\g + \log 8\pi)\nK)
+ \Iint(F) \nK
- \Jint(F) r_1.
\end{multline}
Hence~\eqref{theoeq} can be stated as
\begin{equation}\label{equivtheo}
2\int_{0}^{+\infty}F(x)\cosh\Big(\frac{x}{2}\Big)\d x
>
\sum_\g \int_{0}^{+\infty}F(x)\cos(x\g) \d x.
\end{equation}

Let $\Phi$ be an even, integrable and compactly supported function, and let $F=\Phi\ast\Phi$. Then
\begin{align*}
\int_{0}^{+\infty} F(x)\cosh\Big(\frac{x}{2}\Big)\d x
  &= 2\Big(\int_{0}^{+\infty}\Phi(x)\cosh\Big(\frac{x}{2}\Big)\d x\Big)^2, \\
\int_{0}^{+\infty} F(x)\cos(xt)\d x
  &= 2\Big(\int_{0}^{+\infty}\Phi(x)\cos(xt)\d x\Big)^2,
\end{align*}
and $F$ satisfies~\eqref{equivtheo} and hence~\eqref{theoeq} if and only if
\begin{equation}\label{eqzero}
8\Big(\int_{0}^{+\infty}\Phi(x)\cosh\Big(\frac{x}{2}\Big)\d x\Big)^2
>
\sum_\g\widehat\Phi(\g)^2,
\end{equation}
where we have set $\widehat\Phi(t):=\int_{\RM}\Phi(x)e^{ixt}\d x=2\int_{0}^{+\infty}\Phi(x)\cos(xt)\d x$.

\section{Bounds for class group generators}
Assume $T>1$ and let $L:=\log T$. Let $\Phi^+$ be a real, non-negative, piecewise continuous function with
positively measured support in $[0,L]$, and let
\begin{equation}\label{eq:setup}
\begin{split}
\Phi^-(x)     &:= \Phi^+(-x),                     \\
\Phi^\circ(x) &:= \Phi^+(L/2+x),                  \\
\Phi(x)       &:= \Phi^\circ(x) + \Phi^\circ(-x), \\
F             &:= \Phi\ast\Phi.
\end{split}
\end{equation}
These choices ensure that $F\in\WC(T)$.
\begin{prop}\label{prop:eqzero-simplifiee}
Assume \GRH\ and let $F$ as in~\eqref{eq:setup}. Then~\eqref{theoeq} is satisfied by $F$ if
\begin{multline}\label{equiveqzero-Phi-simplifiee}
\sqrt{T}
\geq
\frac{2T\int_{0}^{L}\big(\Phi^+(x)\big)^2\d x}{\big(\int_{0}^L\Phi^+(x)e^{x/2}\d x\big)^2}
( \lDK
- (\g + \log 8\pi)\nK
)                                                                                          \\
- \frac{4T\sum_{\pG,m}\log\N\pG\frac{(\Phi^+\ast\Phi^-)(m\log\N\pG)}{\N\pG^{m/2}}}{\big(\int_{0}^L\Phi^+(x)e^{x/2}\d x\big)^2}
+ \frac{2T\,\Iint(\Phi^+\ast\Phi^-)\nK}{\big(\int_{0}^L\Phi^+(x)e^{x/2}\d x\big)^2}
+ \frac{2T\int_{0}^L\Phi^+(x)e^{-x/2}\d x}{\int_{0}^L\Phi^+(x)e^{x/2}\d x}.
\end{multline}
\end{prop}
\begin{proof}
We have
\begin{align*}
2\int_{0}^{+\infty}\Phi(x)\cosh\Big(\frac{x}{2}\Big)\d x
%
%
       =  \frac{1}{T^{1/4}}\int_{0}^L\Phi^+(x)e^{x/2}\d x
        +          T^{1/4} \int_{0}^L\Phi^+(x)e^{-x/2}\d x.
\end{align*}
Hence
\begin{multline*}
8\Big(\int_{0}^{+\infty}\Phi(x)\cosh\Big(\frac{x}{2}\Big)\d x\Big)^2\\
>   \frac{2}{\sqrt{T}} \Big(\int_{0}^L\Phi^+(x)e^{x/2}\d x\Big)^2
         + 4\int_{0}^L\Phi^+(x)e^{x/2}\d x\int_{0}^L\Phi^+(x)e^{-x/2}\d x.
\end{multline*}
Moreover
\begin{equation}\label{eq:1c}
\widehat{\Phi}(t)
%
       = 2\Ree\int_{\RM}\Phi^\circ(x)e^{ixt}\d x
%
%
       = 2\Ree\big[e^{-i\frac{Lt}{2}}\widehat{\Phi^+}(t)\big].
\end{equation}
Hence
\begin{equation}\label{eq:2c}
|\widehat{\Phi}(t)|^2
\leq 4 \big|e^{-i\frac{Lt}{2}}\widehat{\Phi^+}(t)\big|^2
=    4 \big|\widehat{\Phi^+}(t)\big|^2.
\end{equation}
We have $\big|\widehat{\Phi^+}(t)\big|^2 = \widehat{\Phi^+}(t) \overline{\widehat{\Phi^+}(t)} =
\widehat{\Phi^+}(t) \widehat{\Phi^-}(t) = \widehat{\Phi^+\ast\Phi^-}(t)$. Thus to
satisfy~\eqref{eqzero} it is sufficient that
\begin{equation}\label{equiveqzero-Phi}
\frac{2}{\sqrt{T}} \Big(\int_{0}^L\Phi^+(x)e^{x/2}\d x\Big)^2
+ 4\int_{0}^L\Phi^+(x)e^{x/2}\d x\int_{0}^L\Phi^+(x)e^{-x/2}\d x
\geq \sum_\g \widehat{\digamma}(\g),
\end{equation}
where
\[
\digamma := 4\Phi^+\ast\Phi^-.
\]
By Weil's Explicit Formula~\eqref{WEF},
\begin{multline}\label{Weil-Phi}
\sum_\g \widehat{\digamma}(\g)
\leq \digamma(0)\big(\lDK - (\g + \log 8\pi)\nK\big)
    - 2\sum_{\pG,m}\log\N\pG\frac{\digamma\big(m\log\N\pG\big)}{\N\pG^{m/2}} \\
    + \Iint(\digamma)\nK
    + 4\int_{0}^{+\infty} \digamma(x)\cosh\Big(\frac{x}{2}\Big)\d x
\end{multline}
where we cancelled the term $-\Jint(\digamma)r_1$ because $\digamma\geq 0$. Notice that
\enlargethispage{2\baselineskip}
\begin{align}
\digamma(0)
&=4\int_{0}^{L}\big(\Phi^+(x)\big)^2\d x                                   \label{bho}
\intertext{and}
\int_{0}^{+\infty} \digamma(x)\cosh\Big(\frac{x}{2}\Big)\d x
&= \frac{1}{2}\int_{\RM}\digamma(x)e^{x/2}\d x                              \nt
&= 2\int_{0}^{L} \Phi^+(x) e^{x/2}\d x \int_{0}^{L} \Phi^+(x) e^{-x/2}\d x. \label{ribho}
\end{align}
The claim follows combining~\eqref{equiveqzero-Phi}, \eqref{Weil-Phi}, \eqref{bho} and~\eqref{ribho}.
\end{proof}

\subsection{Upper bound for $\TeK$}
The coefficient of $\lDK$ in Proposition~\ref{prop:eqzero-simplifiee} is
\[
\frac{2T\int_{0}^{L}\big(\Phi^+(x)\big)^2\d x}{\big(\int_{0}^L\Phi^+(x)e^{x/2}\d x\big)^2}.
\]
Cauchy--Schwarz's inequality shows that its minimum value is $\frac{2T}{T-1}$ and it is attained only for
$\Phi^+(x)=e^{x/2}$ on $[0,L]$. We are interested into small values for this coefficient, hence this is
the best choice we can do. However, this function produces in~\eqref{equiveqzero-Phi-simplifiee} an
inequality for $T$ that cannot be solved easily and, moreover, this choice does not give the best possible
results for secondary coefficients. To overcome this problem in the next theorem we consider the
functions $e^{x/2}\chi_{[L-a,L]}(x)$, where $a$ is a parameter which is fixed in $(0,L]$. This is a
suboptimal choice for the coefficient of $\lDK$ if $a \neq L$, but every value of $a$ independent of $T$
produces an inequality which can be solved easily, still having the correct order for the main term.
Furthermore, acting on $a$ we can minimize also the total contribution coming from the other terms
in~\eqref{equiveqzero-Phi-simplifiee}. Theorem~\ref{theo:Teasynt} is proved using several values of $a$,
and would not be accessible using only the conclusions coming from the choice $a=L$.
\begin{defi}
Assume $a\in(0,L]$. Let $\Phie^+(x):=e^{x/2}\Chi_{[L-a,L]}(x)$ and let $F_{\e}$ be the $F$ defined
in~\eqref{eq:setup} when $\Phi^+=\Phie^+$.
\end{defi}
\begin{rem}\label{rem:Fe}
We recall that $F_{\e}$ is even with support in $[-L,L]$. Moreover, we find that for every $x\in [0,L]$,
\begin{align*}
F_{\e}(x)&= \delta_1(x)(2a-L+x)e^{x/2}\sqrt{T}
         + \delta_2(x)(L-x)e^{x/2}\sqrt{T}  \\
       &\quad
         + 2\delta_3(x)\big(e^{-x/2}-e^{x/2-a}\big)T
         + \delta_4(x)(2a-L-x)e^{-x/2}\sqrt{T}
\end{align*}
where $\delta_1:=\Chi_{[L-2a,L-a)}$, $\delta_2:=\Chi_{[L-a,L]}$, $\delta_3:=\Chi_{[0,a]}$ and
$\delta_4:=\Chi_{[0,2a-L]}$.
\end{rem}

\begin{defi}
Let $\TeK$ be the minimal $T$ such that the function $F_{\e}$ satisfies~\eqref{theoeq} for some $a$.
\end{defi}
Note that $\TCK\leq \TeK$ so that we will state most results about $\TeK$ below as results on $\TCK$ as
well.
\begin{theo}\label{theo:Phie}
Assume \GRH. Fix $T_0>1$. We then have
\begin{align}\label{eq:Phielong}
\sqrt{\TeK}
&\leq \max\Big(\sqrt{T_0},
               r(\lDK,\nK,T_0)
                  - \frac{4}{\big(1-T_0^{-1}\big)^2}\sum_{\N\pG^m\leq T_0}\Big(\frac{1}{\N\pG^m}-\frac{1}{T_0}\Big)\log\N\pG
          \Big),
\intertext{in particular}
\label{eq:Phie}
\sqrt{\TeK}
&\leq \max\big(\sqrt{T_0},r(\lDK,\nK,T_0)\big),
\end{align}
where
\[
r(\Ell,n,t) := \frac{2}{1-t^{-1}}\Big(\Ell
                                    + \log t
                                    - \Big(\g
                                         + \log 2\pi
                                         -  \frac{\log t}{t-1}
                                         +  \log\big(1-t^{-1}\big)\Big)n\Big).
\]
\end{theo}
\begin{proof}
We are assuming $\Phi^+=\Phie^+$ for some $a\leq L$. In this case we have
\begin{align*}
\int_{0}^{L}\Phie^+(x)e^{x/2}\d x
&= \big(1-e^{-a}\big)T
 = \int_{0}^{L}\big(\Phie^+(x))^2\d x, \\
\int_{0}^{L}\Phie^+(x)e^{-x/2}\d x
&= a.
\end{align*}
Moreover, $\forall x\in[0,a]$,
\begin{align*}
\Phie^+\ast\Phie^-(x)
 &= \big(e^{-x/2}-e^{x/2-a}\big)T.
\end{align*}
In addition $\forall x>a$, $\Phie^+\ast\Phie^-(x)=0$. This means that
\begin{align*}
\Iint(\Phie^+\ast\Phie^-)
  &  =  \int_{0}^{+\infty} \frac{\Phie^+\ast\Phie^-(0)-\Phie^+\ast\Phie^-(x)}{2\sinh(x/2)}\d x \\
  &  =  (\log 4)\big(1-e^{-a}\big)T
     +  aT
     -  \big(1-e^{-a}\big)\log\big(e^{a}-1\big)T.
\end{align*}
We now set $a=:\log T_0$ for some $T_0>1$. Since we need to have $L=\log T\geq a$, we get
that~\eqref{equiveqzero-Phi-simplifiee} is satisfied for any $T\geq T_0$ such that
{\small
\begin{multline*}
\sqrt{T}
\geq                                                                                                  \\
\frac{2}{1-T_0^{-1}}\Big(\lDK
- \frac{2}{1-T_0^{-1}}\sum_{\N\pG^m\leq T_0}\Big(\frac{\log\N\pG}{\N\pG^m}-\frac{\log\N\pG}{T_0}\Big)
- \Big(\g + \log 2\pi
               -  \frac{\log T_0}{T_0-1}
               +  \log\big(1-T_0^{-1}\big)\Big)\nK
+ \log T_0\Big).
\end{multline*}
}%
Since the right-hand side does not depend on $T$, this proves the first claim. The second is an obvious
consequence, because the sum on prime ideals is non-negative.
\end{proof}

\subsection{Upper bounds for class group generators}
Theorem~\ref{theo:Teasynt} below gives an upper bound for $\TeK$, and hence for $\TCK$. It is
essentially the best result we can deduce from Theorem~\ref{theo:Phie} (see the remark immediately
following the proof). The theorem has Corollaries~\ref{coro:2} and~\ref{coro:Bach4.01} as easy
consequences.
\begin{theo}\label{theo:Teasynt}
We have
\begin{align*}
\sqrt{\TCK}\leq
\sqrt{\TeK} &\leq 2\Big(\lDK
                      + \llDK
                      - (\g+\log 2\pi)\nK
                      + 1
                      + (\nK+1)\tfrac{\log(7\lDK)}{\lDK}
                   \Big).
\intertext{Moreover, if $\lDK\geq \nK2^{\nK}$, we have}
\sqrt{\TCK}\leq
\sqrt{\TeK} &\leq 2(\lDK
                  + \llDK
                  - (\g+\log 2\pi)\nK
                  + 1).
\end{align*}
\end{theo}
\begin{proof}
We use~\eqref{eq:Phie} with $T_0=\lDK+1$. We have
\begin{align*}
\frac{1}{2}r(\Ell,n,\Ell+1) &= \Ell
                               + \log\Ell
                               - (\g+\log 2\pi)n
                               + 1
                               + (n+1)\frac{\log\Ell}{\Ell}
                               - f(\Ell)n
                               + g(\Ell),
\end{align*}
where
\begin{align*}
f(\Ell) &:= (\g + \log 2\pi)\Ell^{-1}-(1+\Ell^{-1})^2-\Ell^{-2}\log\Ell, \\
g(\Ell) &:= (1+\Ell^{-1})\log(1+\Ell^{-1}).
\end{align*}
We have $f(\Ell)\geq 0$ and $g(\Ell)-2f(\Ell)\leq 0$ for any $\Ell\geq 4$. This proves that
\begin{multline}\label{eq:rlDK+1}
\frac{1}{2}r(\lDK,\nK,\lDK+1) \\
\leq \lDK
  +  \llDK
  -  (\g+\log 2\pi)\nK
  +  1
  +  (\nK+1)\frac{\llDK}{\lDK}
\end{multline}
for any $\K$ such that $\lDK\geq 4$.
%
%
We look under which condition $\frac{1}{2}\sqrt{T_0}$ satisfies the same bound, i.e. when
\begin{equation}\label{eq:sqrtlDK+1}
\frac{1}{2}\sqrt{\lDK+1} \leq \lDK
                            +  \llDK
                            -  (\g+\log 2\pi)\nK
                            +  1
                            +  (\nK+1)\frac{\llDK}{\lDK}.
\end{equation}
Let, for $n\geq 2$ and $\Ell\geq 1$
\[
h(\Ell,n):= \Ell
          + \log\Ell
          - (\g+\log 2\pi)n
          + 1
          + (n+1)\frac{\log\Ell}{\Ell}
          - \frac{1}{2}\sqrt{\Ell+1},
\]
so that~\eqref{eq:sqrtlDK+1} holds true if $h(\lDK,\nK)\geq 0$. We have $\frac{\partial h}{\partial \Ell}
\geq 0$ if $\Ell\geq 0.2n-1$,
%
%
which is true each time $\Ell=\lDK$ and $n=\nK$. This allows to prove that if $\Ell_0\geq 0.2n-1$
satisfies $h(\Ell_0,n)\geq 0$, then $h(\Ell,n)\geq 0$ if $\Ell\geq \Ell_0$. Case $b=2.3$ in Table~3
of~\cite{Odlyzko:tables} shows that $\lDK\geq 2.8\nK-9.6$:
%
%
using this inequality we have $h(\lDK,\nK)\geq 0$ if $\nK\geq 17$.
%
%
For $2\leq \nK\leq 16$, we still have $h(\lDK,\nK)\geq 0$ for $\lDK\geq \LC_0(\nK)$ where $\LC_0(\nK)$ is
indicated in the table below. The table also gives the minimum possible $\lDK$ for the given $\nK$,
computed either with ``megrez'' number field table or with Odlyzko's Table~3.
\[
\begin{array}{rcr|rcr|rcr}
\nK & \min\lDK & \LC_0(\nK) &
\nK & \min\lDK & \LC_0(\nK) &
\nK & \min\lDK & \LC_0(\nK) \\
\hline
  2 &    1.098 &      2.697 &
  7 &   12.125 &     13.676 &
 12 &   24.336 &     25.675 \\
  3 &    3.135 &      4.576 &
  8 &   13.972 &     16.053 &
 13 &   27.749 &     28.096 \\
  4 &    4.762 &      6.728 &
  9 &   17.118 &     18.446 &
 14 &   29.748 &     30.520 \\
  5 &    7.383 &      8.995 &
 10 &   19.060 &     20.849 &
 15 &   33.256 &     32.948 \\
  6 &    9.184 &     11.319 &
 11 &   22.359 &     23.259 &
 16 &   35.277 &     35.378
\end{array}
\]
%
%
We note that~\eqref{eq:sqrtlDK+1} holds also for $\nK=15$. By~\eqref{eq:Phie}, \eqref{eq:rlDK+1}
and~\eqref{eq:sqrtlDK+1}, the first claim is proved for $\nK\geq 17$ or $\lDK\geq \max(4,\LC_0(\nK))$ (in
an even stronger form, because now we have $\log \lDK$ instead of $\log(7\lDK)$).\\
To complete the proof and to extend the claim to $2\leq \nK\leq 16$ and $\lDK\leq \max(4,\LC_0(\nK))$ we
use a different strategy. Let
\[
\Ef(n,t) := \frac{1}{2}(t^{1/2} - t^{-1/2})
          - \log t
          + \Big(\g
               + \log 2\pi
               -  \frac{\log t}{t-1}
               +  \log\big(1-t^{-1}\big)\Big)n,
\]
for $n>0$ and $t>1$. It is the function such that
\[
r(\Ef(n,t),n,t) = \sqrt{t},
\]
hence, if $\lDK = \Ef(\nK,T_0)$, then by~\eqref{eq:Phie} $T_0$ is an upper-bound for $\TeK$; note that
this corresponds to the case $a=L=\log T$ in Theorem~\ref{theo:Phie}. Observe that $\Ef$ is increasing as
a function of $t$, and that it diverges to $-\infty$ and to $+\infty$ for $t\to 1^-$ and $t\to +\infty$,
respectively, for every fixed $n$. As a consequence, for given $\nK$ and $\lDK$ there is a unique $T_0$
such that $\Ef(\nK,T_0)=\lDK$, and this $T_0$ is also an upper-bound for $\TeK$.\\
%
%
%
Thus, for $2\leq \nK\leq 16$ (only a finite set of cases) and $\lDK\leq \max(4,\LC_0(\nK))$ (a bounded
range for $\lDK$) we set $T_0$ such that $\lDK=\Ef(\nK,T_0)$ and we directly check that
\[
\frac{1}{2\sqrt{T_0}}
+ \Big(\frac{\log T_0}{T_0-1} - \log(1-T_0^{-1})\Big)\nK
\leq 1
  +  \log\Big(\frac{\Ef(\nK,T_0)}{T_0}\Big)
  +  (\nK+1)\frac{\log(7\Ef(\nK,T_0))}{\Ef(\nK,T_0)}
\]
%
%
which is equivalent to
\[
\frac{1}{2}
\sqrt{T_0} \leq \lDK
             +  \llDK
             -  (\g+\log 2\pi)\nK
             +  1
             +  (\nK+1)\frac{\log(7\lDK)}{\lDK}.
\]
(Note that now $\log(7\lDK)$ appears, as in the claim.)

For the second claim of the theorem, we use~\eqref{eq:Phielong} still with $T_0=\lDK+1$. We compute a
lower bound for the sum of prime ideals choosing two prime ideals $\pG_0$ and $\pG_1$ above respectively
$2$ and $3$. We get
\begin{align}
\sum_{\N\pG^m\leq T_0}\Big(\frac{1}{\N\pG^m}-\frac{1}{T_0}\Big)\log\N\pG
&\geq \Big(\frac{1}{\N\pG_0}-\frac{1}{T_0}\Big)\log\N\pG_0
   +  \delta_{\nK,2}\Big(\frac{1}{\N\pG_1}-\frac{1}{T_0}\Big)\log\N\pG_1
\notag
\intertext{where $\delta_{\nK,2}$ is $1$ if $\nK=2$ and $0$ otherwise. Note that this holds in any case
because if $\pG_0$ or $\pG_1$ does not appear in the original sum, then the chosen lower bound is
negative. In its turn this is}
&\geq \nK(\log 2)\Big(\frac{1}{2^{\nK}}-\frac{1}{T_0}\Big)
   +  \nK(\log 3)\delta_{\nK,2}\Big(\frac{1}{3^{\nK}}-\frac{1}{T_0}\Big),\label{eq:lowerboundsum}
\end{align}
because the inert case gives the least contribution. Since $\max(4,\LC_0(\nK))\leq \nK2^{\nK}$ for all
$\nK\leq 16$, \eqref{eq:rlDK+1} holds if $\lDK\geq \nK2^{\nK}$. Hence to prove the second claim it is
sufficient to prove that if $\lDK\geq \nK2^{\nK}$, then
\begin{multline}\label{eq:restorlDK+1}
(\nK+1)\frac{\log\lDK}{\lDK}                                                 \\
\leq 2\nK(\log 2)\Big(\frac{1}{2^{\nK}}-\frac{1}{\lDK+1}\Big)
  +   \nK(\log 3)\delta_{\nK,2}\Big(\frac{1}{3^{\nK}}-\frac{1}{\lDK+1}\Big)
\end{multline}
and
\[
\sqrt{\lDK+1} \leq 2(\lDK
                     + \llDK
                     - (\g+\log 2\pi)\nK
                     + 1).
\]
The second statement is elementary and is true for any $\nK\geq 2$.
%
%
For~\eqref{eq:restorlDK+1}, we observe that the left-hand side is decreasing in $\lDK$ while the
right-hand side is increasing, hence it is sufficient to verify it with $\lDK$ substituted by
$\nK2^{\nK}$. One can see that it is true for $\nK\geq 7$ and for $2\leq \nK\leq 6$ and $\lDK\geq
\LC_1(\nK)$ as indicated in table below
\[
\begin{array}{rrr}
\nK & \nK2^{\nK} & \LC_1(\nK) \\
\hline
2 &   8 &  15.670 \\
3 &  24 &  35.173 \\
4 &  64 &  78.801 \\
5 & 160 & 174.859 \\
6 & 384 & 384.395
\end{array}
\]
%
%
To fill the gap, we use~\eqref{eq:Phielong}, \eqref{eq:lowerboundsum} and $T_0=\lDK+7$.
%
%
\end{proof}
\begin{rem*}
The first claim is somehow the best we can hope from~\eqref{eq:Phie}. Indeed the optimal $T_0$
for~\eqref{eq:Phie} is such that
\[
\lDK = T_0 - (\nK+1)\log T_0 + \Big(\g+\log 2\pi - 2\frac{\log T_0}{T_0-1} + \log(1-T_0^{-1})\Big)\nK - 1
\]
for all but a finite number of fields of degree $\nK\leq 22$. Using this formula, one checks that the best
bound we can get from~\eqref{eq:Phie} is
\[
2(\lDK + \llDK - (\g+\log 2\pi)\nK + 1 + \epsilon(\lDK))
\]
where $\epsilon(\lDK)\sim (\nK+1)\llDK/\lDK$ for $\lDK\to\infty$ and fixed $\nK$.
\end{rem*}

\begin{rem*}
Using the full strength of~\eqref{eq:Phielong}, one can prove that for quadratic fields the second claim
of Theorem~\ref{theo:Teasynt} is true for $\TCK$ also for $\lDK\leq \nK2^{\nK}=8$ with only the four
exceptions $\QM[\sqrt{-15}]$, $\QM[\sqrt{-5}]$, $\QM[\sqrt{-23}]$ and $\QM[\sqrt{-6}]$ (for which
$\TCK=2$) and the ten fields of discriminant in $[-11,13]\cup\{-19\}$ (for which the class group is
trivial).
\end{rem*}
We now prove that, for fixed $\nK$, the absolute upper bound for $\TCK/\lDKsq$ is near $4$ and that the
asymptotic limit $4\lDKsq$ is true for a very large set of fields.
\begin{coro}\label{coro:2}
We have
\[
\TCK\leq \TeK \leq 4\big(1+\big(2\pi e^{\g}\big)^{-\nK}\big)^2\lDKsq.
\]
Moreover,
\[
\text{if}\quad
\lDK\leq \frac{1}{e}\big(2\pi e^{\g}\big)^{\nK}\quad\text{then}\quad
\TCK\leq \TeK\leq 4\lDKsq.
\]
\end{coro}
Notice that $2\pi e^\gamma> 11.19$.
\begin{proof}
It is sufficient to prove that $\TeK\leq 4\lDKsq$ if $\lDK\leq \nK2^{\nK}$, because the second statement
of Theorem~\ref{theo:Teasynt} already proves both statements in the remaining ranges.

The right-hand side of the first statement of Theorem~\ref{theo:Teasynt} is $2\lDK+2f(\lDK,\nK)$ with
\[
f(\Ell,n) := \log \Ell
           + 1
           - (\g+\log 2\pi)n
           + (n+1)\frac{\log(7\Ell)}{\Ell}.
\]
We just need to check that $f(\lDK,\nK)\leq 0$ if $\lDK\leq \nK2^{\nK}$. As a function of $\Ell\geq
1$, for fixed $n\geq 2$, $\frac{\partial f}{\partial \Ell}=\frac{n+1}{\Ell}\big(\frac{1}{n+1}
-\frac{\log(7\Ell)-1}{\Ell}\big)$ is negative then positive,
%
%
hence to check that $f$ is negative, it is sufficient to check its value for the minimum and the maximum
$\Ell$ we are interested in. We have $f(n 2^n,n)<0$ for any $n\geq 2$ and $f(\log 23,n)<0$ for any
$n\geq 3$.
%
%
Thus $f(\lDK,\nK)<0$ for any field of degree $\nK\geq 3$. For quadratic fields, we come back
to~\eqref{eq:Phie} with $T_0=2\lDK$ to directly check that $\TeK\leq 4\lDKsq$ if $\lDK\leq
\nK 2^{\nK}=8$.
\end{proof}

\begin{rem*}
The upper bound for the quotient $\TeK/\lDKsq$ tends obviously very fast to $4$. For instance, for
$\nK\geq 10$, we have $\TCK\leq \TeK\leq (4+2.6\cdot 10^{-10})\lDKsq$
%
%
-- but notice that the second claim in Corollary~\ref{coro:2} shows that $\TCK\leq \TeK\leq 4\lDKsq$ if
$\DK\leq \exp(10^{10})$.
%
%
\end{rem*}
For a much tighter range of discriminants one can prove bounds of the form $\TeK\leq c\lDKsq$ with
$c<4$. For instance, we have the psychologically important bound $\TCK\leq \TeK\leq \lDKsq$ as soon as
$$\lDK+2\llDK+2+2(\nK+1)\frac{\log(7\lDK)}{\lDK}\leq 2(\g+\log 2\pi)\nK.$$
For a given degree $\nK\geq 4$, this happens for $\DK$ lower than a certain limit. As $\nK$ goes to
infinity, the limit corresponds to a root-discriminant tending to $(2\pi e^{\g})^2=125.23\dots$.
%
%
%
There are infinitely many fields satisfying this condition. Indeed, consider the field
$F=\QM[\cos(2\pi/11),\sqrt{2},\sqrt{-23}]$. Martinet~\cite{Martinet:tours} proved that the Hilbert
class field tower of $F$ is infinite because $F$ satisfies Golod--Shafarevich's condition. Since
%
$n_F=20$ and $\log\Delta_F\leq 90.6$, this shows that there is an infinite number of fields $\K$ such that
$\lDK\leq 4.53\nK$.
For one of those fields, we have $\TCK\leq \TeK\leq \lDKsq$ if $\nK\geq 47$ and the quotient improves
when the degree increases, with $\limsup\{\TeK/\lDKsq\colon\allowbreak \lDK\leq 4.53\nK\}\leq 0.88$.\\
%
%
As a second example, consider $F=\QM[x]/(f)$, where $f=x^{10} + 223x^8 + 18336x^6 + 10907521x^4 +
930369979x^2 + 18559139599$. Hajir and Maire~\cite{HajirMaire:asymptotically-good} proved that the Hilbert
class field tower of $F$ is infinite because $F$ satisfies Golod--Shafarevich's condition. Since
$\log\Delta_F\leq 44.4$, this shows that there is an infinite number of fields $\K$ such that $\lDK\leq
4.44\nK$.
%
%
For one of those fields, we have $\TCK\leq \TeK\leq \lDKsq$ if $\nK\geq 34$ with $\limsup\{\TeK/\lDKsq
\colon \lDK\leq 4.44\nK\}\leq 0.84$.\\
%
%
\enlargethispage{-\baselineskip}
As a third example, consider the field $F=\QM[x]/(f)$, where $f=x^{12} + 339x^{10} - 19752x^8 -
2188735x^6 + 284236829x^4 + 4401349506x^2 + 15622982921$. In~\cite{HajirMaire:TamelyRamifiedTowers}, the
authors proved that $F$ admits an infinite tower of extensions ramified at most above a single prime ideal
of $F$ of norm $9$. Since $\log(9\Delta_F)/12\leq 4.41$, there is an infinite number of fields $\K$ such
that $\lDK\leq 4.41\nK$.
%
%
For one of those fields, we have $\TCK\leq \TeK\leq \lDKsq$ if $\nK\geq 32$ with $\limsup\{\TeK/\lDKsq
\colon \lDK\leq 4.41\nK\}\leq 0.82$.\\
%
%
Assuming \GRH, Serre~\cite{Serre:discriminants} proved that there are only finitely many
fields such that $\lDK\leq c\nK$ for every $c<\g+\log 8\pi$. Suppose that $\lDK\leq (\g+\log 8\pi)\nK$,
then $\TCK\leq \TeK\leq \lDKsq$ if $\nK\geq 11$. Serre's result does not rule out the possibility that
there are infinitely many such fields; in this case $\limsup \{\TeK/\lDKsq\colon \lDK\leq
(\g+\log 8\pi)\nK\} \leq \big(\frac{4\log 2}{\g+\log 8\pi}\big)^2 \leq 0.54$.
%
%

\begin{coro}\label{coro:Bach4.01}
Assume \GRH. Then
\[
\TCK\leq \TeK\leq 4.01\lDKsq.
\]
\end{coro}
\begin{proof}
For $\nK\geq 3$ or $\nK=2$ and $\lDK\leq (2\pi e^{\g})^{\nK}/e$, the claim follows from
Corollary~\ref{coro:2}.\\
%
%
For $\nK=2$ and $\lDK\geq (2\pi e^{\g})^{\nK}/e$ we apply a different argument. Let
\[
f_\K(n,t) = \log t
          - \Big(\g
               + \log 2\pi
               -  \frac{\log t}{t-1}
               +  \log\big(1-t^{-1}\big)\!\Big)n
          - \frac{2}{1-t^{-1}}\sum_{\N\pG^m\leq t}\Big(\frac{1}{\N\pG^m}-\frac{1}{t}\Big)\log\N\pG,
\]
so that~\eqref{eq:Phielong} can be written as
\[
\sqrt{\TeK} \leq \max\Big(\sqrt{T_0},
                          \frac{2}{1-T_0^{-1}}(\lDK + f_\K(\nK,T_0))
                     \Big).
\]
Suppose we have a $T_0$ such that $f_\K(\nK,T_0)\leq 0$, then we have
\[
\sqrt{\TeK} \leq \max\Big(\sqrt{T_0},
                          \frac{2}{1-T_0^{-1}}\lDK
                     \Big),
\]
and hence
\[
\TeK \leq \frac{4}{(1-T_0^{-1})^2}\lDKsq
\]
if $\lDK\geq\frac{1}{2}(T_0^{1/2}-T_0^{-1/2})$. Recalling that $\nK=2$, we choose $T_0=935$:
%
in this case
\[
f_\K(2,935)\leq 2 - \frac{935}{467}\sum_{\N\pG^m\leq 935}\Big(\frac{1}{\N\pG^m}-\frac{1}{935}\Big)\log\N\pG.
\]
The value of the sum on prime ideals depends on $\K$, but it is always larger than what we get
assuming that all primes are inert. This gives
\[
f_\K(2,935)\leq 2 - \frac{935}{467}\sum_{p^{2m}\leq 935}\Big(\frac{1}{p^{2m}}-\frac{1}{935}\Big)\log(p^2)
\leq -0.02
\]
which therefore produces $\TeK\leq 4.0086\lDKsq$ for $\lDK\geq 15.3$. The proof is complete because
$(2\pi e^{\g})^2/e\geq 46$.
\end{proof}
\subsection{Lower bound for $\Te(\K)$}
\begin{prop}\label{prop:3.9}
Assume \GRH. Then
\[
\sqrt{\TeK} \geq (1 + o(1))\frac{\lDK}{\nK}.
\]
\vspace{-1.5\baselineskip}
\end{prop}
\begin{proof}
Let $S(T)$ denote the Dirichlet series appearing on the left-hand side of~\eqref{theoeq}. Then,
introducing the generalized von Mangoldt function $\tilde{\Lambda}_\K(n) := \sum_{\N\pG^m=n}\log\N\pG$ we get
\begin{equation}\label{eq:1a}
S(T)
= \sum_{n} \frac{2F_{\e}(\log n)}{\sqrt{n}}\tilde{\Lambda}_\K(n).
\end{equation}
Using $\tilde{\Lambda}_\K(n)\leq \nK\Lambda(n)$ in~\eqref{eq:1a} and introducing Stieltjes' integral
notation, we have
\begin{equation*}
\frac{S(T)}{\nK}
\leq \int_{2^-}^{+\infty}  \frac{2F_{\e}(\log x)}{\sqrt{x}}\d\psi(x).
\end{equation*}
Let $g(x)=\frac{2F_{\e}(\log x)}{\sqrt{x}}$ and notice that it is a continuous function which is derivable
except at most in $T^{\pm1}$ and $(Te^{-2a})^{\pm1}$ with a derivative which is continuous where it exists
and bounded. Thus, with a partial integration we get:
\begin{align*}
\frac{S(T)}{\nK} &
 \leq -\int_{2}^{+\infty} g'(x)\psi(x)\d x
 \leq -\int_{2}^{+\infty} g'(x)x\d x
   +   \int_{2}^{+\infty} |g'(x)||\psi(x)-x|\d x                                                   \\
&  =   \int_{2}^{+\infty} g(x)\d x
   +   \int_{2}^{+\infty} |g'(x)||\psi(x)-x|\d x
   +   2g(2).
\intertext{Since under RH $|\psi(x) - x|\leq 2\sqrt{x}\log^2x$ for every $x\geq 2$ (Schoenfeld proved that
RH implies $|\psi(x) -x|\leq \frac{1}{8\pi}\sqrt{x}\log^2x$ as soon as $x\geq 74$, a direct computation
shows that inequality for the intermediate range $x\in[2,74]$) we get}
%
%
&\leq \int_{2}^{+\infty}  g(x)\d x
   +  2\int_{2}^{+\infty} |g'(x)|\sqrt{x}\,\log^2 x\d x
   +  3F_{\e}(\log 2)                                                                              \\
&  =  2\int_{2}^{+\infty} \frac{F_{\e}(\log x)}{\sqrt{x}}\d x
   +  2\int_{2}^{+\infty}\Big|\Big(\frac{F_{\e}(\log x)}{\sqrt{x}}\Big)'\Big|\sqrt{x}\,\log^2x\d x
   +  3F_{\e}(\log 2)                                                                              \\
&  =  2\int_{\log 2}^{+\infty} F_{\e}(x)e^{x/2}\d x
   +   \int_{\log 2}^{+\infty}  |2F_{\e}'(x)-F_{\e}(x)|x^2\d x
   +  3F_{\e}(\log 2).
\end{align*}
We extend the range of the integrals, getting
\begin{equation}\label{eq:1b}
\frac{S(T)}{\nK}
\leq 2\int_{0}^{+\infty} F_{\e}(x)e^{x/2}\d x
   +  \int_{0}^{+\infty} |2F_{\e}'(x)-F_{\e}(x)|x^2\d x
   + 3F_{\e}(\log 2).
\end{equation}
We notice that
\begin{equation}
\int_{0}^{+\infty} F_{\e}(x)e^{x/2}\d x
\leq \int_{\RM} F_{\e}(x)e^{x/2}\d x
 =    \Big(\int_{\RM} \Phie(x)e^{x/2}\d x\Big)^2.                                     \label{eq:2b}
\end{equation}
We observe that, since $\Phie\geq 0$, $\max F_{\e} = F_{\e}(0)$ and that the non-negative part of the
support of $F_{\e}$ is included in $[0,a]\cup[L-2a,L]$ (the intervals may overlap) hence
\begin{align}
\int_{0}^{+\infty}|F_{\e}(x)|x^2\d x
 &= \int_{0}^{L} F_{\e}(x)x^2\d x
 \leq F_{\e}(0)\Big(\int_{0}^a x^2\d x + \int_{L-2a}^{L} x^2\d x\Big) \notag\\
 &  =  F_{\e}(0)\big(2aL^2-4a^2L + 3a^3\big)
  \leq 2aF_{\e}(0)L^2.                                             \label{eq:3b}
\end{align}
Moreover, from Remark~\ref{rem:Fe}, we see that the function is piecewise of the form
$(ax+b)e^{x/2}+(cx+d)e^{-x/2}$, with $a$, $b$, $c$ and $d$ constants, with at most four pieces. Deriving
the expression we find that it can have at most three variations in each piece. The total variation of
$F_{\e}$ on $[0,L]$ is thus at most $12\max F_{\e}=12F_{\e}(0)$. It follows that
\begin{equation}\label{eq:4b}
\int_{0}^{+\infty}|F_{\e}'(x)|x^2\d x
\leq 12F_{\e}(0)L^2.
\end{equation}
Plugging~\eqref{eq:2b},~\eqref{eq:3b} and \eqref{eq:4b} into~\eqref{eq:1b} we get
\begin{align*}
\frac{S(T)}{\nK}
&\leq 2\Big(\int_{\RM} \Phie(x)e^{x/2}\d x\Big)^2
   +  2(a+12)F_{\e}(0)L^2
   +  3F_{\e}(0)                                          \\
&  =  2\Big(e^{-a}T^{3/4}\int_{0}^a e^{x}\d x
          + T^{-1/4}\int_{0}^a \d x\Big)^2
   +  2(a+12)F_{\e}(0)L^2
   +  3F_{\e}(0)                                          \\
&  =  2\Big(\big(1-e^{-a}\big)T^{3/4} + aT^{-1/4}\Big)^2
   +  2(a+12)F_{\e}(0)L^2
   +  3F_{\e}(0)                                          \\
&  =  2\big(1-e^{-a}\big)^2T^{3/2}
   +  4a\big(1-e^{-a}\big)T^{1/2}
   +  a^2T^{-1/2}
   +  2(a+12)F_{\e}(0)L^2
   +  3F_{\e}(0).
\end{align*}
We have $\Jint(F_{\e})\leq\frac{\pi}{2}F_{\e}(0)$ hence in order to satisfy~\eqref{theoeq} we must have
\begin{multline*}
2\big(1-e^{-a}\big)^2T^{3/2}
   + 4a\big(1-e^{-a}\big)T^{1/2}
   + a^2T^{-1/2}
   +  2(a+12)F_{\e}(0)L^2
   +  3F_{\e}(0)                                                                   \\
\geq \frac{S(T)}{\nK}
\geq F_{\e}(0)\Big(\frac{\lDK}{\nK} - \Big(\g + \log8\pi + \frac{\pi}{2}\Big)\Big)
\end{multline*}
that we simplify to
\[
(1+o(1))\sqrt{T}
\geq
\frac{F_{\e}(0)}{2\big(1-e^{-a}\big)^2T}
\Big(\frac{\lDK}{\nK}
   - 2(a+12)L^2
   + O(1)
\Big).
\]
Since $F_{\e}(0)=\int_{\RM}(\Phie(y))^2\d y\geq 2\int_{L/2-a}^{L/2}e^{L/2+y}\d y =2\big(1-e^{-a}\big)T$,
this requires
\[
\big(1+o(1)\big)^2\sqrt{T}
\geq \frac{1}{1-e^{-a}}
\Big(\frac{\lDK}{\nK}
   - 2(a+12)L^2
   + O(1)
\Big).
\]
In the given range for $a$, we can assume that the right-hand side is positive otherwise the claim is
evident. In that case the minimum for the main term is obviously $a=\log T$, hence we can assume that
\[
(1+o(1))\sqrt{T}
\geq \frac{1}{1-T^{-1}}
\Big(\frac{\lDK}{\nK}
    + O(\log^3\lDK)
\Big).
\]
The claim follows.
\end{proof}

\section{Upper bound for $T(\K)$}
Belabas, Diaz y Diaz and Friedman~\cite[Section~3]{small-generators} applied Theorem~\ref{theoKB} with
$F(x)=F_L(x):=(L-x)\chi_{[-L,L]}(x)=(\Phi\ast\Phi)(x)$, where $\Phi$ is the characteristic function of
$[-L/2,L/2]$, with $L=\log T$, $T>1$. (Actually they chose $F=\frac1L\Phi\ast\Phi$, but the difference
does not matter since Equation~\eqref{theoeq} is homogeneous). For this weight function, \eqref{theoeq}
reads
\begin{equation}\label{biribiri}
2\sum_{\substack{\pG,m\\\N\pG^m<T}}\frac{\log\N\pG}{\N\pG^{m/2}}\Big(1-\frac{\log\N\pG^m}{L}\Big)
    >
      \lDK
      - (\g+\log 8\pi)\nK
      + \frac{\Iint(F_L)}{L}\nK
      - \frac{\Jint(F_L)}{L}r_1.
\end{equation}
with
\begin{align*}
\Iint(F_L) &= \frac{\pi^2}2-4\dilog\left(\frac1{\sqrt T}\right)+\dilog\left(\frac1T\right)
\leq \frac{\pi^2}2
\intertext{and}
\Jint(F_L) &= \frac{\pi L}{2}-4C+4\Imm\dilog\left(\frac i{\sqrt T}\right)
\geq \frac{\pi L}{2}-4C,
\end{align*}
where $\dilog x=-\int_0^x\frac{\log(1-u)}{u}\d u$ and $C=\sum_{k\geq 0}(-1)^k(2k+1)^{-2}=0.9159\ldots$ is
Catalan's constant. Note that Belabas, Diaz y Diaz and Friedman use the estimated values for $I(F_L)$
and $J(F_L)$ instead of their exact values: this is a legitimate simplification which affects the
conclusions only by very small quantities.
In this way they produce a quick algorithm giving a bound $T(\K)$ for $\TCK$ which, in explicit
computations, is very small. Unfortunately they also prove that it is $\geq ((\frac{1}{4\nK} +
o(1))\lDK\llDK)^2$, and that therefore it is asymptotically worse than Bach's bound. In the same paper
they advance the conjecture that $(\lDK\llDK)^2$ is the correct size of $T(\K)$, guessing that
$T(\K)=((\frac{1}{4} + o(1))\lDK\llDK)^2$. In this section we prove for $T(\K)$ the corresponding upper
bound and some explicit bounds.
\subsection{Estimation of the number of zeros}
We first prove an estimation of the number of zeros of Dedekind's zeta function that we will use to prove
the main result of this section.
\begin{defi}
Let, for $t\in\RM$,
\[
\NK(t):=\#\{\rho\colon |\g|\leq t\},
\]
where the number is intended including the multiplicity.
\end{defi}
Trudgian~\cite{Trudgian:zero-counting} gives an estimation of $\NK(t)$ for $t\geq 1$. The bound depends
on a parameter $\eta$ which we take equal to $0.05$. In that case the formula is:
\begin{align*}
\forall t\geq 1,\quad
\NK(t)    & =\frac{t}{\pi}\log\Big(\DK\Big(\frac{t}{2\pi e}\Big)^{\nK}\Big)+R_\K(t),
\intertext{with}
|R_\K(t)| &\leq 0.247(\lDK+\nK\log t)+8.851\nK+3.024.
\end{align*}
For $t\in(0,1]$, we use a different strategy.
\begin{prop}\label{N(T)-T<1}
Assume \GRH. We have
\[
\forall t\in(0,1],\quad
\NK(t) \leq  0.637t\Big(\lDK-2.45\nK+S\Big(\frac{3.03}t\Big)\Big)
\]
where
\[
S(U)
:= 960\frac{\big((U-4)e^{\frac{U}{4}}+(U+4)e^{-\frac{U}{4}}\big)^2}{U^5}.
\]
\end{prop}
\begin{proof}
We use an analog of \cite[Proposition~1]{Omar}, but with a different weight: the function
$F:=30\phi\ast\phi$ with $\phi(x):=\big(\frac14-x^2\big)\Chi_{[-\frac12,\frac12]}(x)$.
We then have
\[
F(x) = \begin{array}\{{l@{}>{\quad}l}.
        -|x|^5 + 5|x|^3 - 5x^2 + 1      \case     |x|\leq 1\\
%
%
        0                               \case   1<|x|
        \end{array}
\]
and
\begin{align*}
\widehat F(t)
&= 30\big(\widehat\phi(t)\big)^2
 = 120\frac{\big(2\sin\big(\frac{t}{2}\big)-t\cos\big(\frac{t}{2}\big)\big)^2}{t^6}.
\end{align*}
Observe that $\widehat F$ is decreasing on $[0,8.98]$. Consider, for $U>0$,
%
%
\[
F_U(x):=F\Big(\frac{x}{U}\Big).
\]
We then have
\[
\widehat{F_U}(t)=U\widehat F(Ut)
                 =120U\frac{\big(2\sin\big(\frac{Ut}2\big)-Ut\cos\big(\frac{Ut}2\big)\big)^2}{(Ut)^6}.
\]
Using Weil's Explicit Formula~\eqref{WEF} we have
\begin{multline}\label{eqfromWEF}
U\sum_\g\widehat F(U\g) = 4\int_{0}^{+\infty} F_U(x)\cosh\Big(\frac{x}{2}\Big)\d x
        - 2\sum_{\pG,m}\frac{\log\N\pG}{\N\pG^{\frac{m}{2}}}F_U(m\log\N\pG)      \\
        + \lDK
        - (\g + \log8\pi)\nK
        + \Iint(F_U) \nK
        - \Jint(F_U) r_1.
\end{multline}
We have
\begin{align*}
4\int_{0}^{+\infty} F_U(x)\cosh\Big(\frac{x}{2}\Big)\d x
&= 4\int_{0}^{+\infty} F\Big(\frac{x}{U}\Big)\cosh\Big(\frac{x}{2}\Big)\d x \\
&= 4U\int_{0}^1(1-5x^2+5x^3-x^5)\cosh\Big(\frac{Ux}{2}\Big)\d x             \\
&= 960\frac{\big((U-4)e^{\frac U4}+(U+4)e^{-\frac U4}\big)^2}{U^5}
 = S(U).
\end{align*}
Moreover,
\begin{align*}
\Iint(F_U)  & = \int_{0}^{+\infty}\frac{1-F_U(x)}{2\sinh(x/2)}\d x
              = U\int_{0}^{+\infty}(1-F(x))\frac{e^{-Ux/2}}{1-e^{-Ux}}\d x.
\intertext{Integrating by parts, which is possible because $F$ is $C^2$, it becomes}
%
        & = 5\int_{0}^1(2x-3x^2+x^4)\log\Big(\frac{1+e^{-Ux/2}}{1-e^{-Ux/2}}\Big)\d x
\end{align*}
from which we readily see that $\Iint(F_U)$ is decreasing. Removing the positive terms $\sum_{\pG,m}$ and
$\Jint(F_U)r_1$ from~\eqref{eqfromWEF}, we get
\[
\forall U>0,\quad U\sum_\g\widehat F(U\g) \leq  \lDK-(\g+\log8\pi-\Iint(F_U))\nK+S(U).
\]
Let $t\in(0,1]$ and $c$ such that $0<c\leq 8.98$, then setting $U=\frac{c}{t}$ and using
$\Iint(F_{c/t})\leq \Iint(F_c)$ we have
\[
\sum_\g\widehat F\Big(\frac{c\g}{t}\Big)
\leq
\frac{t}{c}\Big(\lDK-(\g+\log8\pi-\Iint(F_c))\nK+S\Big(\frac{c}{t}\Big)\Big)
\]
and
\[
\widehat F(c)\NK(t)
\leq
\sum_{|\g|\leq t}\widehat F\Big(\frac{c\g}{t}\Big)
\leq
\sum_\g\widehat F\Big(\frac{c\g}{t}\Big)
\]
so that $\forall t\in(0,1]$, $\forall c\in(0,8.98]$ we have
\[
\NK(t)
\leq
\frac{c^5 t}{120(2\sin(c/2)-c\cos(c/2))^2}\Big(\lDK-\big(\g + \log8\pi - \Iint(F_c)\big)\nK + S\Big(\frac{c}{t}\Big)\Big).
\]
The value of $c$ minimizing the coefficient of $t\lDK$ is $3.051\dots$.
%
%
The claim follows setting $c=3.03$.
%
%
\end{proof}
\begin{defi}
Let $\MK(t)$ be the function
\[
\MK(t) :=
\begin{cases}
0.637t\big(\lDK-2.45\nK+S\big(\frac{3.03}{t}\big)\big) \case0< t<1 \\
\frac{t}{\pi}\log\big(\DK\big(\frac{t}{2\pi e}\big)^{\nK}\big)
            + 0.247(\lDK+\nK\log t)+8.851\nK+3.024      \case1\leq t.
\end{cases}
\]
\end{defi}
Recalling the previous proposition we thus have
\[
\forall t> 0,\quad\NK(t)\leq \MK(t).
\]

\subsection{The bound}
Now we are in position to prove the announced upper bounds for $T(\K)$.
\begin{theo}\label{estim-T(K)}
Assume \GRH. We have for any fixed $\nK$
\[
\limsup_{\DK\to\infty} \frac{T(\K)}{\big(\lDK\llDK\big)^2}\leq \frac{1}{16}.
\]
Moreover, for any field
$\K\not\in\big\{\QM[\sqrt{-1}],\QM[\sqrt{-3}],\QM[\sqrt5]\big\}$
we have
\[
T(\K)\leq 3.9\big(\lDK\llDK\big)^2.
\]
\end{theo}
\begin{rem*}
Computing $T(\K)$ for the whole ``megrez'' number field table~\cite{MegrezTables}, we find that the
quotient $T(\K)/(\lDK\llDK)^2$ is mostly $\leq0.27$ for them. In fact, apart the fields appearing
as exceptions in the theorem and for which the quotient is $\geq10$, there are only six more fields for
which it is $\geq1$: the quadratic fields of discriminant in $\{-11,-8,-7,8,12,13\}$.
%
%
\end{rem*}
\begin{proof}
Assume $T>e$, $L=\log T$ and $F:=F_L=\Phi\ast\Phi$ with $\Phi:=\Chi_{[-L/2,L/2]}$. Then~\eqref{eqzero}
becomes
\begin{equation*}
4\Big(\sqrt{T} - 2 + \frac1{\sqrt{T}}\Big)
                   \geq  \sum_\g\frac{1-\cos(L\g)}{\g^2},
\end{equation*}
but to estimate $T(\K)$ we have to add the function
\[
G(T) = \Big(4\dilog\Big(\frac{1}{\sqrt{T}}\Big) - \dilog\Big(\frac{1}{T}\Big)\Big)\nK
     + 4\Imm\dilog\Big(\frac{i}{\sqrt{T}}\Big)r_1
\]
to the right-hand side, as a consequence of the approximations used in~\cite{small-generators} for
$\Iint(F_L)$ and $\Jint(F_L)$. Thus, introducing $f(t):=\frac{1-\cos t}{t^2}$, the condition becomes
\begin{equation}\label{condition-equivalente-2-BDyDF}
4\Big(\sqrt{T} - 2 + \frac1{\sqrt{T}}\Big) \geq  L^2\sum_\g f(L\g) + G(T),
\end{equation}
and we need an upper bound for the sum appearing on the right-hand side.
Function $G(T)$ may be easily estimated, for $T\geq e$, as
\begin{equation}\label{eq:G}
G(T) \leq \frac{8\nK}{\sqrt{T}}\sum_{k=0}^\infty \frac{T^{-2k}}{(4k+1)^2}
     \leq \frac{8.05\nK}{\sqrt{T}}.
\end{equation}
%
The main contribution to the sum on zeros comes from those which are close to $0$, the remaining ones
being easily and quite well estimated via the partial summation formula. Thus, we consider first the
range $|t|\leq 1$. The best absolute bound for $f(Lt)$ in this range is $\tfrac12 L^2$, and if we bound
the sum $\sum_{|\g|\leq}f(L\g)$ simply as $\sup_{t\in[0,1]}|f(Lt)|\NK(1)$ then we get a term of size
$\frac12 L^2\lDK$. With this bound~\eqref{condition-equivalente-2-BDyDF} would become
\[
(4 + o(1))\sqrt{T}> \Big(\frac{1}{2} + o(1)\Big) L^2\lDK
\]
forcing $T$ to $\geq (\frac{1}{4} +o(1))(\lDK)^2(\log\lDK)^4$, which is much larger than what we want to
prove.\\
We overcome this problem using the conclusion of Proposition~\ref{N(T)-T<1} to bound $\NK(t)$ for $t\leq
1$, but in itself this is not yet sufficient since that bound diverges as $t$ goes to $0$. Hence, for
very small $\g$ we apply a more involved argument. In a way similar to the proof of
Proposition~\ref{N(T)-T<1} but using $F:=\Phi\ast\Phi$ with $\Phi:=\Chi_{[-1/2,1/2]}$, we have,
\begin{align}
\forall U&>0,&
2U\sum_\g f(U\g)
&\leq \lDK
      - (\g + \log8\pi - \IF(U))\nK
      + \frac{8}{U}\big(e^{U/4}-e^{-U/4}\big)^2,         \notag
\intertext{where $\IF(U) := \Iint(F_U) = \frac{1}{U}\big(\frac{\pi^2}{2}+4\dilog(e^{-U/2})
- \dilog(e^{-U})\big)$. Hence}
\forall U&\geq 3.545,&
2U\sum_\g f(U\g)
&\leq \lDK - 2.6016\,\nK + \frac{8}{U}e^{\frac{U}{2}}.   \label{majoration-en-U}
%
%
%
\end{align}
Setting $U=L$, this bound gives immediately a bound for $L^2\sum_\g f(L\g)$ of the right order
$\frac{1}{2}L\lDK$ for the part depending on the discriminant. Unfortunately, it also contains the term
$4e^{\frac{L}{2}} = 4\sqrt{T}$ which makes the bound completely useless when inserted
in~\eqref{condition-equivalente-2-BDyDF}. As a consequence we have to modify a bit this approach, and we
use~\eqref{majoration-en-U} with $U:=L-2\log L$. In fact, we notice that $\forall t\geq 0$, $f'(t)\leq
0.014$.
%
%
This means that
\begin{equation}\label{eq:delta-bound}
\forall t\geq 0,\quad
f(Lt)   \leq  f((L-2\log L)t) + 0.028t\log L.
\end{equation}
Below $\gimel:=\big(0.014L^2\log L\big)^{-1/3}$ we bound $f(Lt)$ using~\eqref{eq:delta-bound} otherwise we
use the trivial bound $f(Lt)\leq \frac2{(Lt)^2}$. We thus define
\[
\forall t\geq 0,
\quad
g(t) :=
\begin{cases}
  0.028t L^2\log L     \case 0\leq t\leq \gimel \\
  \frac2{t^2}          \case \gimel<t.
\end{cases}
\]
Note that we have chosen $\gimel$ in such a way that $g$ is continuous. With this definition of $g$, we
thus have $L^2f(Lt)\leq L^2f((L-2\log L)t)\Chi_{[0,\gimel]}(t)+g(t)$.\\
Since we use~\eqref{majoration-en-U} with $U=L-2\log L$, we need that $L-2\log L\geq 3.545$ so that we
suppose $T\geq 2000$. This, in turn, means that $\gimel\leq 1$.
%
%
In this way we get
\begin{equation}\label{partiel}
L^2\sum_\g f(L\g)
\leq
L^2\sum_{|\g|\leq \gimel}f((L-2\log L)\g)
  + \sum_\g g(\g).
\end{equation}
By~\eqref{majoration-en-U}, for the first part we have
\begin{align}
L^2\sum_{|\g|\leq \gimel}f((L-2\log L)\g)
\leq
\frac{L^2}{2(L-2\log L)}\Big(\lDK-2.6016\nK+\frac{8\sqrt{T}}{L(L-2\log L)}\Big). \label{prim-part}
\end{align}
Notice that in this way the term containing $\sqrt{T}$ is actually $O(\sqrt{T}/\log T)$
and does not interfere any more.\\
We now estimate the second part of~\eqref{partiel}. We have
\begin{align}
\sum_\g g(\g)
&  =  \sum_{|\g|\leq \gimel}g(\g)
   +  \sum_{\gimel<|\g|}g(\g)
 \leq g(\gimel)\NK(\gimel)
   +  \int_{\gimel^+}^{+\infty} g(t)\d\NK(t)
   =  -\int_\gimel^{+\infty} g'(t)\NK(t)\d t.                                               \notag
\intertext{Since $g'$ is non-positive on $[\gimel,\infty)$ we can estimate $\NK$ by
$\MK$, getting}
&\leq -\int_\gimel^{+\infty} g'(t)\MK(t)\d t
   =  -\int_\gimel^1g'(t)\MK(t)\d t
   -  \int_{1}^{+\infty} g'(t)\MK(t)\d t.                                                   \notag
\intertext{The first integral can be estimated by noticing that $S$ is increasing so that
$\MK(t)\leq \frac{t}{\gimel}\MK(\gimel)$ in $[\gimel,1]$. The second integral can be
computed. In this way we have}
\sum_\g g(\g)
&\leq 2.55\Big(\lDK-2.45\nK+S\Big(\frac{3.03}\gimel\Big)\Big)\int_\gimel^1\frac{\!\d t}{t^2}
%
%
   +  g(1)\MK(1)
   +  \int_{1}^{+\infty} g(t)F'_\K(t)\d t                                                   \nt
&  =  2.55\Big(\lDK-2.45\nK+S\Big(\frac{3.03}\gimel\Big)\Big)\Big(\frac1\gimel-1\Big)         \nt
&\quad
  + \frac2\pi(\lDK-\nK\log(2\pi e))
  + 2\cdot 0.247\lDK
  + 2(8.851\nK+3.024)                                                                         \nt
&\quad
   +  2\int_{1}^{+\infty}\Big(\frac1\pi\Big(\lDK+\nK\log\big(\frac{t}{2\pi}\big)\Big)+\frac{0.247\nK}t\Big)\frac{\d t}{t^2} \nt
&  =  2.55\Big(\lDK-2.45\nK+S\Big(\frac{3.03}\gimel\Big)\Big)\Big(\frac1\gimel-1\Big) \label{second-part}\\
&\quad
  + \frac4\pi(\lDK-\nK\log2\pi)
  + 0.247(2\lDK+\nK)
  + 2(8.851\nK+3.024).                                                                        \notag
\end{align}
Inserting~\eqref{prim-part} and~\eqref{second-part} in~\eqref{partiel} we finally obtain
\begin{align*}
L^2\sum_\g f(L\g)
&\leq
\frac{L^2}{2(L-2\log L)}\Big(\lDK-2.6016\nK+\frac{8\sqrt{T}}{L(L-2\log L)}\Big)               \\
&\quad
  + 2.55\Big[\lDK{-}2.45\nK{+}S\Big(3.03\big(0.014L^2\log L\big)^{1/3}\Big)\Big]
  \Big[\big(0.014L^2\log L\big)^{1/3}{-}1\Big]                                                \\
&\quad
  + \frac4\pi(\lDK-\nK\log2\pi)
  + 0.247(2\lDK+\nK)
  + 2(8.851\nK+3.024)                                                                         \\
&\leq
\Big(\frac{L^2}{2(L-2\log L)}
  + 0.62\big(L^2\log L\big)^{1/3}
  - 0.78\Big)\lDK
%
%
  + \frac{4L}{(L-2\log L)^2}\sqrt{T}                                                          \\
&\quad
  + \Big(-\frac{1.3008L^2}{L-2\log L}
%
%
  - 1.5057\big(L^2\log L\big)^{1/3}
%
%
  + 21.857\Big)\nK                                                                            \\
%
%
&\quad
  + 2.55\,
    S\Big(3.03\big(0.014L^2\log L\big)^{1/3}\Big)\Big(\big(0.014L^2\log L\big)^{1/3} - 1\Big)
  + 6.05.
%
%
\end{align*}
We are assuming $T\geq 2000$, thus the coefficient of $\nK$ is smaller than $3.17-1.3L$
%
%
and we get
\begin{multline}\label{derniere}
L^2\sum_\g f(L\g)
\leq
\Big(\frac{L^2}{2(L{-}2\log L)}
  {+} 0.62\big(L^2\log L\big)^{\frac13}
  {-} 0.78\Big)\lDK
  {+} \frac{4L}{(L{-}2\log L)^2}\sqrt{T}                                      \\
%
%
  {+} 2.55\,
    S\Big(3.03\big(0.014L^2\log L\big)^{\frac13}\Big)\Big(\big(0.014L^2\log L\big)^{1/3} {-} 1\Big)
  {+} (3.17 {-} 1.3L)\nK
  {+} 6.05.
%
%
\end{multline}
We can now deal with the first part of the proposition, that is
\[
\limsup_{\DK\to\infty} \frac{T(\K)}{\big(\lDK\llDK\big)^2}\leq \frac1{16}
\]
inserting~\eqref{derniere} in~\eqref{condition-equivalente-2-BDyDF}, with the bound~\eqref{eq:G}.
Obviously, $T$ diverges as $\DK$ goes to infinity; in particular, the restriction $T\geq 2000$ does not
matter. Moreover, we observe that the second term in the first line of~\eqref{derniere} is $o(\log
T)\lDK$ and that the second line is $O(\sqrt{T}/\log T)$, hence the first claim is proved.

We now study the second part of the proposition, which means the claimed inequality
\[
T(\K)\leq 3.9\big(\lDK\llDK\big)^2.
\]
We note that, to have $T<2000$ in~\eqref{condition-equivalente-2-BDyDF} with~\eqref{eq:G}
and~\eqref{derniere}, we need $\lDK<5.15 + 0.63\nK$.
%
%
According to Table~3 in~\cite{Odlyzko:tables} (entry $b=1.3$), this may happen only if $\nK\leq 5$
%
%
and also in this case, only when $\DK\leq 607$ (resp. $1141$, $2143$, $4023$) for fields of degree $2$
(resp. $3$, $4$, $5$).\\
%
%
For fields of degree $\nK\geq 6$ and $\DK\geq 1.7\cdot10^5$, by elementary arguments one sees
from~\eqref{condition-equivalente-2-BDyDF}, \eqref{eq:G} and~\eqref{derniere} that
\[
T(\K)\leq 3.6\big(\lDK\llDK\big)^2.
\]
%
%
According to Table~3 in~\cite{Odlyzko:tables}, this covers in particular all fields with degree $\nK\geq
7$.\\
For fields of degree $\nK\leq 5$, we see that
\begin{equation}\label{eq:beppe}
T(\K)\leq 3.9\big(\lDK\llDK\big)^2
\end{equation}
as soon as $\DK\geq 3\cdot10^6$ for quadratic fields,
%
%
or $\DK\geq 10^6$ for $3\leq \nK\leq 5$.
%
%
There remains a finite number of fields of degree $2\leq \nK\leq 6$. All those with $\nK\geq 3$ appear in
``megrez'' number field table~\cite{MegrezTables} and for all fields, including the quadratic ones,
we use the algorithm indicated in \cite{small-generators} as implemented in~\cite{PARI2}. For
$\K\in\big\{\QM[\sqrt{-3}],\QM[\sqrt{-1}]\big\}$ we find $T(\K)=5$ and for $\K=\QM[\sqrt5]$, $T(\K)=7$.
All other fields satisfy~\eqref{eq:beppe}.
%
%
\end{proof}

\section{Multi-step}
\subsection{Bounds for two and three steps}
The original choice $F=\Phi\ast\Phi$ with $\Phi$ the characteristic function of $[-L/2,L/2]$ is of the
type considered in~\eqref{eq:setup} with $\Phi^+(x)=\chi_{[0,L]}(x)$, i.e. a function assuming only one
value in $[0,L]$. We call this choice the \emph{one-step} case. The following corollary shows that the
performance of the algorithm significantly improves already when $\Phi^+(x)$ is allowed to assume two or
three values in the $[0,L]$ interval, as long as it is zero when $x$ is close to $L/2$ (so that
$\Phi(x)=0$ if $x$ is close to $0$). In particular, the extra factor $\llDK$ disappears. We call these
choices \emph{two-} and \emph{three-steps}.
\begin{coro}\label{coro:boundT1cst}
Let $\Phi^+(x)=b\Chi_{[L-2a,L]}(x)+ \Chi_{[L-a,L]}(x)$ and define $F$ as in~\eqref{eq:setup}. Denote
respectively $T_{\mathrm{c}2}(\K)$ (`two steps') and $T_{\mathrm{c}3}(\K)$ (`three steps') the lowest $T$
such that~\eqref{theoeq} is satisfied by such an $F$, with respectively $b=0$ and $b\neq 0$.
We have
\[
\sqrt{T_{\mathrm{c}2}(\K)}\leq \max\big(2.456\lDK - 5.623\nK + 14,\sqrt{13}\big),
\]
where the result is obtained with $a=2.5$, and
%
\[
\sqrt{T_{\mathrm{c}3}(\K)}\leq \max\big(2.193\lDK - 6.19\nK + 16,\sqrt{32}\big),
\]
which is obtained with $a=1.722$ and $b=\frac{e^{-a/2}}{1-e^{-a/2}}$.
\end{coro}
\begin{proof}
The claim is proved directly applying Proposition~\ref{prop:eqzero-simplifiee} with
$\Phi^+(x)=b\Chi_{[L-2a,L]}(x)+ \Chi_{[L-a,L]}(x)$. Conditions $\sqrt{T_{c2}}\geq \sqrt{13}$ and
$\sqrt{T_{c3}}\geq \sqrt{32}$ come from the need to ensure that the support of $\Phi^+$ be in $[0,L]$, so
that we need to assume $L>a$ for the two-steps and $L>2a$ for the three-steps, respectively.
\end{proof}
\subsection{The algorithm}\label{subsec:algo}
Our aim is to find a good $T$ for the number field $\K$ as fast as possible exploiting the bilinearity of
the convolution product. We introduce some definitions to make the discussion easier.
\begin{defi}
Let $\SC$ be the real vector space of even and compactly supported step functions and, for $T>1$, let
$\SC(T)$ be the subspace of $\SC$ of functions supported in $[-L/2,L/2]$, with $L=\log T$.
\end{defi}
\begin{defi}
For any integer $N\geq 1$ and positive real $\delta$ we define the subspace
$\SC_d(N,\delta)$ of $\SC(e^{2N\delta})$ made of functions which are constant
$\forall k\in\NM$ on $[k\delta,(k+1)\delta)$.
\end{defi}
The elements of $\SC_d(N,\delta)$ are thus step functions with fixed step width
$\delta$. If $N\geq 1$, $\delta>0$ and $T=e^{2N\delta}$ we have
\begin{subequations}
\begin{alignat}{1}
\SC_d(N,\delta)                                      & \subset\SC(T)\subset\SC                 \\
\forall\Phi\in\SC(T),\qquad\Phi\ast\Phi              & \in\WC(T)                               \\
\SC_d(N,\delta)                                      & \subset\SC_d(N+1,\delta)\label{incnext} \\
\forall k\geq 1,\quad\SC_d\Big(kN,\frac\delta k\Big) & \subseteq\SC_d(N,\delta)\label{incmul}.
\end{alignat}
\end{subequations}

If, for some $T>1$, $\Phi\in\SC(T)$ and $F=\Phi\ast\Phi$ satisfies \eqref{theoeq} then, according to
Theorem~\ref{theoKB}, $\TCK<T$. This leads us to define the linear form $\ell_\K$ on
$\bigcup_{T>1}\WC(T)$ by
\[
\ell_\K(F) = -2\sum_\pG\log\N\pG\sum_{m=1}^{+\infty}\frac{F(m\log\N\pG)}{\N\pG^{m/2}}
           + F(0)(\lDK - (\g + \log 8\pi)\nK)
           + \Iint(F) \nK
           - \Jint(F) r_1
\]
and the quadratic form $q_\K$ on $\SC$ by $q_\K(\Phi)=\ell_\K(\Phi\ast\Phi)$. From Theorem~\ref{theoKB}
we deduce the following consequence.
\begin{coro}\label{coroKBL}
Let $\K$ be a number field satisfying \GRH\ and $T>1$. If the restriction of $q_\K$ to $\SC(T)$ has a
negative eigenvalue then $\TCK<T$.
\end{coro}
\begin{defi}
A \emph{bound} for $\K$ is an $L=\log T$ with $T$ as in Theorem~\ref{theoKB}.
\end{defi}
Note that $q_\K$ is a continuous function as a function from $(\SC(T),\|.\|_1)$ to $\RM$. Therefore if $L$
is a bound for $\K$ then there exists an $L'<L$ such that $L'$ is a bound for $\K$. Note also that, in
terms of $T$, only the norms of prime ideals are relevant, which means that we do not need the smallest
possible $T$ to get the best result.
\begin{rem}\label{rem:varepsilon}
If $T>1$ and $\Phi\in\SC(T)$, then for any $\varepsilon>0$ there exists $N\geq 1$, $\delta>0$ and
$\Phi_\delta\in\SC_d(N,\delta)$ such that $\|\Phi\ast\Phi-\Phi_\delta\ast\Phi_\delta\|_\infty\leq
\varepsilon$ and $e^{2N\delta}\leq T$. Hence we do not loose anything in terms of bounds for $\K$ if we
consider only the subspaces of the form $\SC_d(N,\delta)$.
\end{rem}
As we will see later, we can compute $q_\K(\Phi)$ for a generic $\Phi\in \SC(T)$ combining its values for
$\Phi=\chi_{[-L/2,L/2]}$ at different $L$'s. Thus, let~$\GRHcheck(\K,L)$ be the function that returns the
right-hand side of~\eqref{biribiri} minus its left-hand side (without the approximations for $\Iint(F_L)$
and $\Jint(F_L)$), and $\BDyDF(\K)$ be the function which implements the algorithm
of~\cite[Section~3]{small-generators}. The computation of $\BDyDF(\K)$ is very fast because the only
arithmetic information we need on $\K\simeq\QM[x]/(P)$ is the splitting information for primes $p<T$ and
is determined easily for nearly all $p$. Indeed if $p$ does not divide the index of $\ZM[x]/(P)$ in $\OK$,
then the splitting of $p$ in $\K$ is determined by the factorization of $P\mod p$. We can also store such
splitting information for all $p$ that we consider and do not recompute it each time we test whether a
given $L$ is a bound for $\K$.

We denote $q_{\K,N,\delta}$ the restriction of $q_\K$ to $\SC_d(N,\delta)$. According to
Corollary~\ref{coroKBL}, if $q_{\K,N,\delta}$ has a negative eigenvalue then $2N\delta$ is a bound for
$\K$. This justifies the following definition.
\begin{defi}
The pair $(N,\delta)$ is \good{} when $q_{\K,N,\delta}$ has a negative eigenvalue.
\end{defi}

We can reinterpret Functions~\GRHcheck and \BDyDF saying that if $\GRHcheck(\K,2\delta)$ is negative then
$(1,\delta)$ is \good{} and that $\big(1,\frac12\log\BDyDF(\K)\big)$ is \good.

The fundamental step for our algorithm is the following: given $\delta>0$ we look for the smallest $N$
such that $(N,\delta)$ is \good. Looking for such an $N$ can be done fairly easily with this setup. For
any $i\geq 1$, let $\Phi_i$ be the characteristic function of $({-i\delta},i\delta)$. Then
$(\Phi_i)_{1\leq i\leq N}$ is a basis of $\SC_d(N,\delta)$. We have
$\Phi_i\ast\Phi_i=F_{2i\delta}=(2i\delta-|x|)\chi_{[-2i\delta,2i\delta]}(x)$. We observe that
$$\Phi_i\ast\Phi_j=F_{(i+j)\delta}-F_{|i-j|\delta}\ .$$
This means that the matrix $A_N$ of $q_{\K,N,\delta}$ can be computed by computing only the values of
$\ell_\K(F_{i\delta})$ for $1\leq i\leq 2N$ and subtracting those values.

We then stop when the determinant of $A_N$ is negative or when $2N\delta\geq \BDyDF(\K)$. This does not
guarantee that we stop as soon as there is a negative eigenvalue. Indeed, consider the following sequence
of signatures:
$$(0,p,0)\to(1,p,0)\to(1,p,1)\to(0,p+1,2)\to\cdots$$
here a signature is $(z,p,m)$ where $z$ is the dimension of the kernel and $p$ (resp. $m$) the dimension
of a maximal subspace where $q_\K$ is positive (resp. negative) definite. We should have stopped when the
signature was $(1,p,1)$ however the determinant was zero there. Our algorithm will stop as soon as there
is an odd number of negative eigenvalues (and no zero) or we go above $\BDyDF(\K)$. Such unfavorable
sequence of signatures is however very unlikely and does not happen in practice.

The corresponding algorithm is presented in Function~\NDelta. We have added a limit $N_{\max}$ for $N$
which is not needed right now but will be used later. Note that $(\Phi_i)$ is a basis adapted to the
inclusion~\eqref{incnext} so that we only need to compute the edges of the matrix $A_N$ at each step. The
test $\det A<0$ in line~\ref{line:detA<0} can be efficiently implemented using Cholesky $LDL^*$
decomposition because it is incremental; moreover, if the last coefficient of $D$ is negative, the last
line of $L^{-1}$ is a vector $v$ such that $vA{}^tv<0$ so that we can check the result.

One way to use this function is to compute $T=\BDyDF(\K)$ and for some $N_{\max}\geq 2$, let
$\delta=\frac{L}{2N_{\max}}$ and $N=\NDelta(\K,\delta,N_{\max})$. Using the inclusion~\eqref{incmul}, we
see that $(N,\delta)$ is \good{} and that $N\leq N_{\max}$, so that we have improved the bound.
\enlargethispage{\baselineskip}
\subsection{Adaptive steps}
Unfortunately Function~\NDelta is not very efficient mostly for two reasons. To explain them and to
improve the function we introduce some extra notations.\\ For any $\delta>0$, let $N_\delta$ be the
minimal $N$ such that $(N,\delta)$ is \good. Observe that Function~\NDelta computes $N_\delta$, as long as
$N_\delta\leq N_{\max}$ and no zero eigenvalue prevents success. Obviously, using~\eqref{incnext}, we see
that for any $N\geq N_\delta$, $(N,\delta)$ is \good. We have observed numerically that the sequence
$N\delta_N$ is roughly decreasing, i.e. for most values of $N$ we have $N\delta_N\geq
(N+1)\delta_{N+1}$.\\
For any $N\geq 1$, let $\delta_N$ be the infimum of the $\delta$'s such that $(N,\delta)$ is \good. It is
not necessarily true that if $\delta\geq \delta_N$ then $(N,\delta)$ is \good, however we have never found
a counterexample. The function $\delta\mapsto\delta N_\delta$ is piecewise linear with discontinuities at
points where $N_\delta$ changes; the function is increasing in the linear pieces and decreasing at the
discontinuities. This means that if we take $0<\delta_2<\delta_1$ but we have $N_{\delta_2}>N_{\delta_1}$
then we may have $\delta_2N_{\delta_2}>\delta_1N_{\delta_1}$ so the bound we get for $\delta_2$ is not
necessarily as good as the one for $\delta_1$.\\
The resolution of Function~\NDelta is not very good: going from $N-1$ to $N$ the bound for the norm of the
prime ideals is multiplied by $e^{2\delta}$. This is the first reason reducing the efficiency of the
function. The second one is that if $N_{\max}$ is above $20$ or so, the number
$\delta=\frac{\log\BDyDF(\K)}{2N_{\max}}$ has no specific reason to be near $\delta_{N_\delta}$; as
discussed above, this means that we can get a better bound for $\K$ by choosing $\delta$ to be just above
either $\delta_{N_\delta}$ or $\delta_{1+N_\delta}$. Both reasons derive from the same facts and give a
bound for $\K$ that can be overestimated by at most $2\delta$ for the considered
$N=\NDelta(\K,\delta,N_{\max})$.

To improve the result, we can use once again inclusion~\eqref{incmul} and determine a good approximation
of $\delta_N$ for $N=2^n$. We determine first by dichotomy a $\delta_0$ such that $(N_0,\delta_0)$ is
\good{} for some $N_0\geq 1$. For any $k\geq 0$, we take $N_{k+1}=2N_k$ and determine by dichotomy a
$\delta_{k+1}$ such that $(N_{k+1},\delta_{k+1})$ is \good; we already know that $\frac{\delta_k}2$ is an
upper bound for $\delta_{k+1}$ and we can either use $0$ as a lower bound or try to find a lower bound not
too far from the upper bound because the upper bound is probably not too bad. The algorithm is described
in Function~\Bound. It uses a subfunction $\OptimalT(\K,N,T_\ell,T_h)$ which returns the smallest integer
$T\in[T_\ell,T_h]$ such that $\NDelta(\K,L/(2N),N)>0$. The algorithm does not return a bound below those
proved in~\ref{theo:Teasynt} and~\ref{coro:Bach4.01}.
\subsection{Further refinements}\label{subsec:algoend}
To improve the speed of the algorithm, we decided to make the dichotomy in~$\OptimalT(\K,N,T_\ell,T_h)$,
not on all value of $T$ but only on the norms of the prime ideals in $[T_\ell,T_h]$.\\
To reduce the time used to compute the determinants, we tried to use steps of width $4\delta$ in
$[-L/2,L/2]$ and of width $2\delta$ in the rest of $[-3L/4,3L/4]$, to halve the dimension of
$\SC_d(N,\delta)$. It worked in the sense that we found substantially the same $T$ faster. However we
decided that the total time of the algorithm is not high enough to justify the increase in code
complexity.
\begin{function}
  \SetAlgoLined
  \KwIn{a number field $\K$}
  \KwIn{a positive real $\delta$}
  \KwIn{a positive integer $N_{\max}$}
  \KwOut{an $N\leqslant N_{\max}$ such that $(N,\delta)$ is \good{} or $0$}

  $tab\leftarrow\text{$(2N_{\max}+1)$-dimensional array}$\;
  $tab[0]\leftarrow0$\;
  $A\leftarrow\text{$N_{\max}\times N_{\max}$ identity matrix}$\;
  $N\leftarrow 0$\;
  \While{$N<N_{\max}$}{
    $N\leftarrow N+1$\;
    $tab[2N-1]\leftarrow(2N-1)\GRHcheck(\K,(2N-1)\delta)$\;
    $tab[2N]\leftarrow2N\GRHcheck(\K,2N\delta)$\;
    \For{$i\leftarrow1$ \KwTo $N$}{
      $A[N,i]\leftarrow tab[N+i]-tab[N-i]$\;
      $A[i,N]\leftarrow A[N,i]$\;
    }
    \If{$\det A<0$}{\label{line:detA<0}
      \KwRet{$N$}\;
    }
  }
  \KwRet{$0$}\;
  \caption{NDelta($\K$,$\delta$,$N_{\max}$)}
\end{function}
\begin{function}
  \SetAlgoLined
  \KwIn{a number field $\K$}
  \KwOut{a bound for the norm of a system of generators of $\Cl_\K$}
  \uIf{$\lDK< \nK 2^{\nK}$}{
      $T_0\leftarrow 4\Big(\lDK+\llDK-(\g+\log 2\pi)\nK+1+(\nK+1)\frac{\log(7\lDK)}\lDK\Big)^2$\;
  }
  \Else{
      $T_0\leftarrow 4(\lDK+\llDK-(\g+\log 2\pi)\nK+1)^2$\;
  }
  $T_0\leftarrow \min\big(T_0,4.01\lDKsq\big)$\;
  $N\leftarrow8$;
  $\delta\leftarrow0.0625$\;
  \While{$\NDelta(\K,\delta,N)=0$}{
    $\delta\leftarrow\delta+0.0625$\;
  }
  $T_h\leftarrow \OptimalT(\K,N,e^{2N\,(\delta-0.0625)},e^{2N\,\delta})$\;
  $T\leftarrow T_h+1$\;
  \While{$T_h<T\mathop{||}T>T_0$}{
    $T\leftarrow T_h$;
    $N\leftarrow2N$\;
    $T_h\leftarrow \OptimalT(\K,N,1,T_h)$\;
  }
  \KwRet{$T$}\;
  \caption{Bound($\K$)}
\end{function}
\clearpage
\subsection{Theoretical performance}
We denote $T_1(\K)$ the result of Function~\Bound. The algorithm reaches bounds of the same quality as
those of $\TeK$.
\begin{coro}\label{coro:boundT1}
Assume \GRH. Then Function~\Bound terminates. Moreover we have
\begin{align*}
\sqrt{T_1(\K)}&\leq 2\lDK
                 +  2\llDK
                 +  2
                 -  2(\g+\log 2\pi)\nK
                 +  2(\nK+1)\frac{\log(7\lDK)}{\lDK},                                       \\
\sqrt{T_1(\K)}&\leq 2\lDK
                 +  2\llDK
                 +  2
                 -  2(\g+\log 2\pi)\nK\qquad\text{if }\lDK\geq \nK 2^{\nK},                 \\
T_1(\K)       &\leq 4\big(1+\big(2\pi e^{\g}\big)^{-\nK}\big)^2\lDKsq,                      \\
T_1(\K)       &\leq 4\lDKsq\qquad\text{if }\lDK\leq \frac{1}{e}\big(2\pi e^{\g}\big)^{\nK}, \\
T_1(\K)       &\leq 4.01\lDKsq.
\end{align*}
\end{coro}
\begin{proof}
Consider one of the bounds of Theorem~\ref{theo:Teasynt} or of Corollary~\ref{coro:Bach4.01}. It is
associated to a certain $F_{\e}=\Phie\ast\Phie$ with a certain $a=\log T_0$ and $T$ given by the bound.
As in Remark~\ref{rem:varepsilon}, For any $\varepsilon>0$, there exists a step function $\Phi$ with
support in $[-L/2,L/2]$ and $2^N$ steps, for $N$ large enough, such that
$\|F-\Phi\ast\Phi\|_\infty<\varepsilon$. Since the inequality in~\eqref{theoeq} is strict, we can take
$\varepsilon$ small enough so that $\Phi\ast\Phi$ satisfies~\eqref{theoeq}. Hence, for $N$ large enough,
the algorithm will find a negative eigenvalue in $\SC_d(2^N,2^{-N-1}L)$ and hence it will terminate. The
bound $T$ it will give satisfies obviously the first two and last inequality of the statement of the
corollary. Since the intermediate inequalities are consequences of the first two, $T$ will also satisfy
the intermediates inequalities.
\end{proof}

\subsection{Effective performance}\label{subsec:tests}
\subsubsection{Various fields}
We tested the algorithm on several fields. Let first $\K=\QM[x]/(P)$ where
$$\catcode`\*=\active\def*{}
P=x^3 + 559752270111028720*x + 55137512477462689.
$$
The polynomial $P$ has been chosen so that for all primes $2\leq p\leq 53$ there are two prime ideals of
norms $p$ and $p^2$. This ensures that there are lots of small norms of prime ideals. We have
$T(\K)=19162$. There are $2148$ non-zero prime ideals with norms up to $T(\K)$. We found that
$T_1(\K)=11071$ and that there are $1343$ non-zero prime ideals of norms up to $T_1(\K)$.

We tested also the algorithm on the set of $4686$ fields of degree $2$ to $27$ and small discriminant
coming from a benchmark of~\cite{PARI2}. The mean value of $T_1(\K)/T(\K)$ for those fields is lower than
$1/2$.

We have tested the cyclotomic fields $\K=\QM[\zeta_n]$ for $n\leq 250$. For them we have found that the
quotient $T_1(\K)/T(\K)$ becomes smaller and smaller as the degree increases, reaching the value $1/2$ for
the higher order cyclotomic fields. However, we have observed that the fraction is generally higher than
what we get for the generic fields with comparable degree and discriminant. Certainly cyclotomic fields
are not typical fields: for instance for them the class number grows more than exponentially as a
function of the order~\cite[Th.~4.20]{Washington1}. The weight function that we observe in the tests for
generic fields contains several parameters and therefore could have generic profiles but actually always
shows two bumps, one centered at the origin and one near the end of the support. On the contrary, the
weight producing $T(\K)$ has a unique bump in the origin, by design. The fact that the original algorithm
already produces a good bound for cyclotomic of small order in some sense means that the second bump is
not necessary, and this is probably due to the existence of a lots of ideals of small norms. However, we
admit we do not have any convincing explanation of this phenomenon.

\subsubsection{Pure fields, small discriminants}
We computed $T(\K)$ and $T_1(\K)$ for fields of the form $\QM[x]/(P)$ with $P=x^n\pm p$ and $p$ is the
first prime after $2^a$ for a certain family of integers $n$ and $a$ such that $\lDK\leq 250$. We limited
the discriminant because, while at the time of writing the record for which the Buchmann algorithm has
been successfully completed has $\lDK\geq 646$, this has been done for only very few fields with $\lDK\geq
100\log 10\simeq 230$. We computed the family of $\frac{T_1(\K)}{T(\K)}$ for each fixed degree. The
results are presented in Figure~\ref{fig1}. We can see that in the right-half of the graph, the fields
adopt the asymptotical behavior where $\frac{T_1(\K)}{T(\K)}\asymp(\llDK)^{-2}$.\\
Let $t(\K)$ denote the time needed to compute $T(\K)$ and $t_1(\K)$ the additional time needed to compute
$T_1(\K)$. In Figure~\ref{fig2}, we have drawn the points $\frac{t_1(\K)}{t(\K)}$ for the four families of
fields we have tested. We have removed three points with $\frac{t_1(\K)}{t(\K)}\geq 25$ (in details:
$30.17$, $30.31$ and $47.19$) whose $\lDK$ is respectively $160.81$, $162.20$ and $167.74$.
%
%
In spite of the relatively large value for the quotient, the value of $t_1(\K)$ in all cases has been
lower than $2$s
(and actually larger than $0.35$s in only $22$ out of the $8308$ fields, including the three fields for
which $t_1(\K)/t(\K)\geq 25$, all having $\lDK\geq 153$).
This shows that the time needed to compute $T_1(\K)$ is limited anyway with respect to the time needed for
the full Buchmann algorithm.
\subsubsection{Pure fields}
We once again computed $T(\K)$ and $T_1(\K)$ for fields of the form $\QM[x]/(P)$ with $P=x^n\pm p$ and $p$
is the first prime after $10^a$ for a certain family of integers $n$ and $a$.
The graph of $\frac{T_1(\K)}{T(\K)}$ looks like a continuation of the right-half of Figure~\ref{fig1} so
that we do not draw it once again. The graph of $\frac{T_1(\K)}{T(\K)}(\llDK)^2$ is much more regular and
looks to have a non-zero limit, see Figure~\ref{fig3} below. We plotted the graph of
$\frac{T_1(\K)}{\lDKsq}$ for the same fields in Figure~\ref{fig4} as well. We computed the mean of
$\frac{T_1(\K)}{T(\K)}(\llDK)^2$ and the maximum of $\frac{T_1(\K)}{\lDKsq}$ for each fixed degree. The
results are summarized below:
\[
\begin{array}{l|r|r|c|c}
P        & a\leq & \lDK\leq & \text{mean of }\frac{T_1(\K)}{T(\K)}(\llDK)^2 & 1-\max\big(\frac{T_1(\K)}{\lDKsq}\big) \\
\hline
x^2-p    & 3999  &  9212    & 13.19 & 2\cdot 10^{-5} \\
x^6+p    & 1199  & 13818    & 13.38 & 9\cdot 10^{-6} \\
x^{21}-p &  328  & 15169    & 13.68 & 4\cdot 10^{-5}
\end{array}
\]
%
The small discriminants are (obviously) much less sensitive to the new algorithm. We reduced the range for
each series to have $\lDK\leq 500$. The results are as follows:
\[
\begin{array}{l|r|c|c}
P        & a\leq & \text{mean of }\frac{T_1(\K)}{T(\K)}(\llDK)^2 & 1-\max\big(\frac{T_1(\K)}{\lDKsq}\big) \\
\hline
x^2-p    &   218 & 12.35 & 0.018 \\
x^6+p    &    43 & 13.66 & 0.073 \\
x^{21}-p &    10 & 17.19 & 0.279
\end{array}
\]
%
\subsubsection{Biquadratic fields}
We repeated the computations above also for biquadratic fields $\QM[\sqrt{p_1},\sqrt{p_2}]$ where each
$p_i$ is the first prime after $2^{a_i}$ (respectively $10^{a_i}$) for certain families of integers $a_i$
and included them in Figures~\ref{fig1}--\ref{fig4}.\\
In the case where $p_i$ is the first prime after $10^{a_i}$, we found that the mean of
$\frac{T_1(\K)}{T(\K)}(\llDK)^2$ is $13.63$ for the $7119$ fields computed, and $13.88$ if we restrict the
family to the $1537$ ones with $\lDK\leq 500$, while the maximum of $\frac{T_1(\K)}{\lDKsq}$ is lower than
$1.0038$ for all fields and $0.957$ for the fields with $\lDK\leq 500$.\\
%
%
\begin{fixedfig}
\input{T-small-relative-result.tex}\\
\caption{$\frac{T_1(\K)}{T(\K)}$ for some fields of small discriminant; in abscissa $\lDK$.}%
\label{fig1}
\end{fixedfig}
\enlargethispage{4\baselineskip}
\begin{fixedfig}
\input{T-small-relative-time.tex}\\
\caption{$\frac{t_1(\K)}{t(\K)}$ for some fields of small discriminant; in abscissa $\lDK$.}%
\label{fig2}%
\end{fixedfig}
\begin{fixedfig}
\input{T-small-fields.tex}\\
\caption{$\frac{T_1(\K)}{T(\K)}(\llDK)^2$ for some fields; in abscissa $\lDK$.}%
\label{fig3}%
\end{fixedfig}
\begin{fixedfig}
\input{T-small-fields-beppe.tex}\\
\caption{$\frac{T_1(\K)}{\lDKsq}$ for some fields; in abscissa $\lDK$.}%
\label{fig4}%
\end{fixedfig}

\subsection{A simplified algorithm}
Since the expression of $F_{\e}$ as given in Remark~\ref{rem:Fe} is simpler for $a=L/2$, we can implement a
variant of the algorithm of Belabas, Diaz y Diaz and Friedman with that weight. We have for $x\geq 0$,
\[
F_{\e}(x) = \Big(\big(x-2+2e^{-x}\sqrt{T}\big)\chi_{[0,L/2)}(x)+(L-x)\chi_{[L/2,L]}(x)\Big)e^{x/2}\sqrt{T}
\]
so that~\eqref{theoeq} becomes
\begin{multline*}
\sum_{\N\pG^m<\sqrt{T}}\Big(m\log\N\pG-2+2\frac{\sqrt{T}}{\N\pG^m}\Big)\log\N\pG
+ \sum_{\sqrt{T}\leq \N\pG^m<T}(L-m\log\N\pG)\log\N\pG                             \\
>
\big(\sqrt{T}-1\big)(\lDK - (\g + \log 8\pi)\nK)
+ \frac{\Iint(F_{\e})}{2\sqrt{T}} \nK
- \frac{\Jint(F_{\e})}{2\sqrt{T}} r_1
\end{multline*}
where
\begin{align*}
\frac{\Iint(F_{\e})}{2\sqrt{T}}
  &= \big(\sqrt{T}-1\big)\log\Big(\frac{4}{1-T^{-1/2}}\Big)-\frac{L^2}{8}+\frac{L}{2}
   - \frac{\pi^2}{12}-\dilog(-T^{-1/2}) \\
\frac{\Jint(F_{\e})}{2\sqrt{T}}
  &= \big(\sqrt{T}+1\big)\log\Big(\frac{2}{1+T^{-1/2}}\Big)+\frac{L^2}{8}-\frac{L}{2}
   - \frac{\pi^2}{24}-\dilog(-T^{-1/2})+\frac{1}{2}\dilog(-T^{-1}).
\end{align*}
Applying Theorem~\ref{theo:Phie} with $T_0=\sqrt{T}$, we can see that the result $T_2(\K)$ of this
algorithm satisfies
\[
\sqrt{T_2(\K)}\leq 2\lDK + 2\llDK - (\g+\log 2\pi)\nK + 1 + 2\log 2 + c\nK\frac{\llDK}{\lDK}
\]
for some absolute constant $c$. This means that the asymptotical expansion is nearly optimal: the first
term that changes with respect to Theorem~\ref{theo:Teasynt} is the constant term which increases from $2$
to $1+2\log 2\simeq 2.38$. For small discriminants this algorithm is sometimes worse than~\BDyDF and
always significantly worse than~\Bound, given in Subsections~\ref{subsec:algo}--\ref{subsec:algoend}. For
larger discriminants, it gives only sightly bigger results than~\Bound but is always faster. Numerically,
in our experiments with the above mentioned fields, $T_2(\K)\geq T(\K)$ if $\lDK\leq 48$ (resp. $142$,
$83$, $162$) for $\nK=2$ (resp. $4$, $6$, $21$) for a total of $401$ fields out the $12648$ tested.
\section{A final comment}
We have proved that
\[
(1+o(1))\frac{\lDK}{\nK}\leq \sqrt{\TeK}\leq (2+o(1))\lDK.
\]
We have three reasons to believe that the ``true'' behavior is
\[
\sqrt{\TeK}\sim \lDK
\]
as $\DK\to\infty$, for fixed $\nK$.

The first one is computational. We have observed that $\sqrt{T_1(\K)}/\lDK$ seems to tend to $1$ (from
below, see Figure~\ref{fig4}) for several series of pure fields and one series of biquadratic fields. We
also tested some restricted cases with $\Phie^+$ and $a=L$ and the $\Phi^+$ of the form indicated in
Corollary~\ref{coro:boundT1cst} with $a=\log 4$ and $b=0$: in all cases we observed the same phenomenon,
which is that the experimental result seems to be half the one we can prove.

The second one is related to the upper bound. The function $\widehat F(t)=4(\Ree[e^{-iLt/2}
\widehat{\Phi^+}(t)])^2$ in~\eqref{eq:1c} is estimated with
$\widehat\digamma(t)=4|\widehat{\Phi^+}(t)|^2$ in~\eqref{eq:2c}. This step removes the quick oscillations
of $e^{-iLt/2}$ and allows the conclusion of the argument, but it overestimates the contribution of this
object, which would be of this size only in the case where the $\g$'s were placed very close to the
maxima of $\cos^2(Lt)$ and which would be smaller by a factor $1/2$, in mean, for uniformly spaced zeros.
Unfortunately, the actual information we have for the vertical distribution of zeros is not strong enough
to distinguish between these two behaviors.

The third one is related to the lower bound in Proposition~\ref{prop:3.9}. For its computation we have
considered each prime integer as totally split. This allows an explicit bound, but it should be corrected
by a factor $1/\nK$, because this is the density of the totally split primes, and the contribution of the
other primes should be negligible.

All three reasons indicate that $1\cdot \lDK$ should be the correct asymptotic. We are unable to prove it,
though.
%


\def\PARI{PARI}\def\megrez{megrez}
\providecommand{\bysame}{\leavevmode\hbox to3em{\hrulefill}\thinspace}
\providecommand{\MR}{\relax\ifhmode\unskip\space\fi MR }
\providecommand{\MRhref}[2]{%
  \href{http://www.ams.org/mathscinet-getitem?mr=#1}{#2}
}
\providecommand{\href}[2]{#2}

\end{document}

%% file: T-small-relative-result.tex
\begingroup
  \makeatletter
  \providecommand\color[2][]{%
    \GenericError{(gnuplot) \space\space\space\@spaces}{%
      Package color not loaded in conjunction with
      terminal option `colourtext'%
    }{See the gnuplot documentation for explanation.%
    }{Either use 'blacktext' in gnuplot or load the package
      color.sty in LaTeX.}%
    \renewcommand\color[2][]{}%
  }%
  \providecommand\includegraphics[2][]{%
    \GenericError{(gnuplot) \space\space\space\@spaces}{%
      Package graphicx or graphics not loaded%
    }{See the gnuplot documentation for explanation.%
    }{The gnuplot epslatex terminal needs graphicx.sty or graphics.sty.}%
    \renewcommand\includegraphics[2][]{}%
  }%
  \providecommand\rotatebox[2]{#2}%
  \@ifundefined{ifGPcolor}{%
    \newif\ifGPcolor
    \GPcolorfalse
  }{}%
  \@ifundefined{ifGPblacktext}{%
    \newif\ifGPblacktext
    \GPblacktexttrue
  }{}%
  \let\gplgaddtomacro\g@addto@macro
  \gdef\gplbacktext{}%
  \gdef\gplfronttext{}%
  \makeatother
  \ifGPblacktext
    \def\colorrgb#1{}%
    \def\colorgray#1{}%
  \else
    \ifGPcolor
      \def\colorrgb#1{\color[rgb]{#1}}%
      \def\colorgray#1{\color[gray]{#1}}%
      \expandafter\def\csname LTw\endcsname{\color{white}}%
      \expandafter\def\csname LTb\endcsname{\color{black}}%
      \expandafter\def\csname LTa\endcsname{\color{black}}%
      \expandafter\def\csname LT0\endcsname{\color[rgb]{1,0,0}}%
      \expandafter\def\csname LT1\endcsname{\color[rgb]{0,1,0}}%
      \expandafter\def\csname LT2\endcsname{\color[rgb]{0,0,1}}%
      \expandafter\def\csname LT3\endcsname{\color[rgb]{1,0,1}}%
      \expandafter\def\csname LT4\endcsname{\color[rgb]{0,1,1}}%
      \expandafter\def\csname LT5\endcsname{\color[rgb]{1,1,0}}%
      \expandafter\def\csname LT6\endcsname{\color[rgb]{0,0,0}}%
      \expandafter\def\csname LT7\endcsname{\color[rgb]{1,0.3,0}}%
      \expandafter\def\csname LT8\endcsname{\color[rgb]{0.5,0.5,0.5}}%
    \else
      \def\colorrgb#1{\color{black}}%
      \def\colorgray#1{\color[gray]{#1}}%
      \expandafter\def\csname LTw\endcsname{\color{white}}%
      \expandafter\def\csname LTb\endcsname{\color{black}}%
      \expandafter\def\csname LTa\endcsname{\color{black}}%
      \expandafter\def\csname LT0\endcsname{\color{black}}%
      \expandafter\def\csname LT1\endcsname{\color{black}}%
      \expandafter\def\csname LT2\endcsname{\color{black}}%
      \expandafter\def\csname LT3\endcsname{\color{black}}%
      \expandafter\def\csname LT4\endcsname{\color{black}}%
      \expandafter\def\csname LT5\endcsname{\color{black}}%
      \expandafter\def\csname LT6\endcsname{\color{black}}%
      \expandafter\def\csname LT7\endcsname{\color{black}}%
      \expandafter\def\csname LT8\endcsname{\color{black}}%
    \fi
  \fi
    \setlength{\unitlength}{0.0500bp}%
    \ifx\gptboxheight\undefined%
      \newlength{\gptboxheight}%
      \newlength{\gptboxwidth}%
      \newsavebox{\gptboxtext}%
    \fi%
    \setlength{\fboxrule}{0.5pt}%
    \setlength{\fboxsep}{1pt}%
\begin{picture}(8502.00,3968.00)%
    \gplgaddtomacro\gplbacktext{%
      \csname LTb\endcsname%
      \put(594,440){\makebox(0,0)[r]{\strut{}$0$}}%
      \put(594,766){\makebox(0,0)[r]{\strut{}$0.1$}}%
      \put(594,1093){\makebox(0,0)[r]{\strut{}$0.2$}}%
      \put(594,1419){\makebox(0,0)[r]{\strut{}$0.3$}}%
      \put(594,1745){\makebox(0,0)[r]{\strut{}$0.4$}}%
      \put(594,2072){\makebox(0,0)[r]{\strut{}$0.5$}}%
      \put(594,2398){\makebox(0,0)[r]{\strut{}$0.6$}}%
      \put(594,2724){\makebox(0,0)[r]{\strut{}$0.7$}}%
      \put(594,3050){\makebox(0,0)[r]{\strut{}$0.8$}}%
      \put(594,3377){\makebox(0,0)[r]{\strut{}$0.9$}}%
      \put(594,3703){\makebox(0,0)[r]{\strut{}$1$}}%
      \put(789,220){\makebox(0,0){\strut{}$0$}}%
      \put(2120,220){\makebox(0,0){\strut{}$50$}}%
      \put(3451,220){\makebox(0,0){\strut{}$100$}}%
      \put(4783,220){\makebox(0,0){\strut{}$150$}}%
      \put(6114,220){\makebox(0,0){\strut{}$200$}}%
      \put(7445,220){\makebox(0,0){\strut{}$250$}}%
      \put(7577,440){\makebox(0,0)[l]{\strut{}$0$}}%
      \put(7577,766){\makebox(0,0)[l]{\strut{}$0.1$}}%
      \put(7577,1093){\makebox(0,0)[l]{\strut{}$0.2$}}%
      \put(7577,1419){\makebox(0,0)[l]{\strut{}$0.3$}}%
      \put(7577,1745){\makebox(0,0)[l]{\strut{}$0.4$}}%
      \put(7577,2072){\makebox(0,0)[l]{\strut{}$0.5$}}%
      \put(7577,2398){\makebox(0,0)[l]{\strut{}$0.6$}}%
      \put(7577,2724){\makebox(0,0)[l]{\strut{}$0.7$}}%
      \put(7577,3050){\makebox(0,0)[l]{\strut{}$0.8$}}%
      \put(7577,3377){\makebox(0,0)[l]{\strut{}$0.9$}}%
      \put(7577,3703){\makebox(0,0)[l]{\strut{}$1$}}%
    }%
    \gplgaddtomacro\gplfronttext{%
      \csname LTb\endcsname%
      \put(5861,1273){\makebox(0,0)[l]{\strut{}quadratic}}%
      \csname LTb\endcsname%
      \put(5861,1053){\makebox(0,0)[l]{\strut{}biquadratic}}%
      \csname LTb\endcsname%
      \put(5861,833){\makebox(0,0)[l]{\strut{}degree 6}}%
      \csname LTb\endcsname%
      \put(5861,613){\makebox(0,0)[l]{\strut{}degree 21}}%
    }%
    \gplbacktext
    \put(0,0){\includegraphics{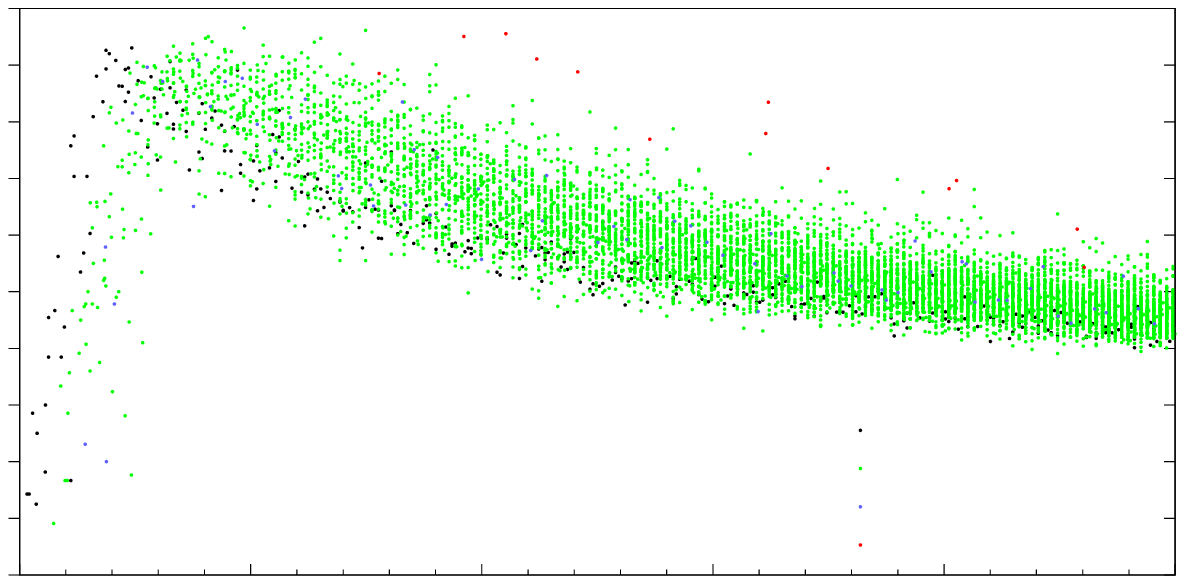}}%
    \gplfronttext
  \end{picture}%
\endgroup

%% file: T-small-relative-time.tex
\begingroup
  \makeatletter
  \providecommand\color[2][]{%
    \GenericError{(gnuplot) \space\space\space\@spaces}{%
      Package color not loaded in conjunction with
      terminal option `colourtext'%
    }{See the gnuplot documentation for explanation.%
    }{Either use 'blacktext' in gnuplot or load the package
      color.sty in LaTeX.}%
    \renewcommand\color[2][]{}%
  }%
  \providecommand\includegraphics[2][]{%
    \GenericError{(gnuplot) \space\space\space\@spaces}{%
      Package graphicx or graphics not loaded%
    }{See the gnuplot documentation for explanation.%
    }{The gnuplot epslatex terminal needs graphicx.sty or graphics.sty.}%
    \renewcommand\includegraphics[2][]{}%
  }%
  \providecommand\rotatebox[2]{#2}%
  \@ifundefined{ifGPcolor}{%
    \newif\ifGPcolor
    \GPcolorfalse
  }{}%
  \@ifundefined{ifGPblacktext}{%
    \newif\ifGPblacktext
    \GPblacktexttrue
  }{}%
  \let\gplgaddtomacro\g@addto@macro
  \gdef\gplbacktext{}%
  \gdef\gplfronttext{}%
  \makeatother
  \ifGPblacktext
    \def\colorrgb#1{}%
    \def\colorgray#1{}%
  \else
    \ifGPcolor
      \def\colorrgb#1{\color[rgb]{#1}}%
      \def\colorgray#1{\color[gray]{#1}}%
      \expandafter\def\csname LTw\endcsname{\color{white}}%
      \expandafter\def\csname LTb\endcsname{\color{black}}%
      \expandafter\def\csname LTa\endcsname{\color{black}}%
      \expandafter\def\csname LT0\endcsname{\color[rgb]{1,0,0}}%
      \expandafter\def\csname LT1\endcsname{\color[rgb]{0,1,0}}%
      \expandafter\def\csname LT2\endcsname{\color[rgb]{0,0,1}}%
      \expandafter\def\csname LT3\endcsname{\color[rgb]{1,0,1}}%
      \expandafter\def\csname LT4\endcsname{\color[rgb]{0,1,1}}%
      \expandafter\def\csname LT5\endcsname{\color[rgb]{1,1,0}}%
      \expandafter\def\csname LT6\endcsname{\color[rgb]{0,0,0}}%
      \expandafter\def\csname LT7\endcsname{\color[rgb]{1,0.3,0}}%
      \expandafter\def\csname LT8\endcsname{\color[rgb]{0.5,0.5,0.5}}%
    \else
      \def\colorrgb#1{\color{black}}%
      \def\colorgray#1{\color[gray]{#1}}%
      \expandafter\def\csname LTw\endcsname{\color{white}}%
      \expandafter\def\csname LTb\endcsname{\color{black}}%
      \expandafter\def\csname LTa\endcsname{\color{black}}%
      \expandafter\def\csname LT0\endcsname{\color{black}}%
      \expandafter\def\csname LT1\endcsname{\color{black}}%
      \expandafter\def\csname LT2\endcsname{\color{black}}%
      \expandafter\def\csname LT3\endcsname{\color{black}}%
      \expandafter\def\csname LT4\endcsname{\color{black}}%
      \expandafter\def\csname LT5\endcsname{\color{black}}%
      \expandafter\def\csname LT6\endcsname{\color{black}}%
      \expandafter\def\csname LT7\endcsname{\color{black}}%
      \expandafter\def\csname LT8\endcsname{\color{black}}%
    \fi
  \fi
    \setlength{\unitlength}{0.0500bp}%
    \ifx\gptboxheight\undefined%
      \newlength{\gptboxheight}%
      \newlength{\gptboxwidth}%
      \newsavebox{\gptboxtext}%
    \fi%
    \setlength{\fboxrule}{0.5pt}%
    \setlength{\fboxsep}{1pt}%
\begin{picture}(8502.00,3968.00)%
    \gplgaddtomacro\gplbacktext{%
      \csname LTb\endcsname%
      \put(462,440){\makebox(0,0)[r]{\strut{}$0$}}%
      \put(462,1093){\makebox(0,0)[r]{\strut{}$5$}}%
      \put(462,1745){\makebox(0,0)[r]{\strut{}$10$}}%
      \put(462,2398){\makebox(0,0)[r]{\strut{}$15$}}%
      \put(462,3050){\makebox(0,0)[r]{\strut{}$20$}}%
      \put(462,3703){\makebox(0,0)[r]{\strut{}$25$}}%
      \put(657,220){\makebox(0,0){\strut{}$0$}}%
      \put(2041,220){\makebox(0,0){\strut{}$50$}}%
      \put(3425,220){\makebox(0,0){\strut{}$100$}}%
      \put(4809,220){\makebox(0,0){\strut{}$150$}}%
      \put(6193,220){\makebox(0,0){\strut{}$200$}}%
      \put(7577,220){\makebox(0,0){\strut{}$250$}}%
      \put(7709,440){\makebox(0,0)[l]{\strut{}$0$}}%
      \put(7709,1093){\makebox(0,0)[l]{\strut{}$5$}}%
      \put(7709,1745){\makebox(0,0)[l]{\strut{}$10$}}%
      \put(7709,2398){\makebox(0,0)[l]{\strut{}$15$}}%
      \put(7709,3050){\makebox(0,0)[l]{\strut{}$20$}}%
      \put(7709,3703){\makebox(0,0)[l]{\strut{}$25$}}%
    }%
    \gplgaddtomacro\gplfronttext{%
    }%
    \gplbacktext
    \put(0,0){\includegraphics{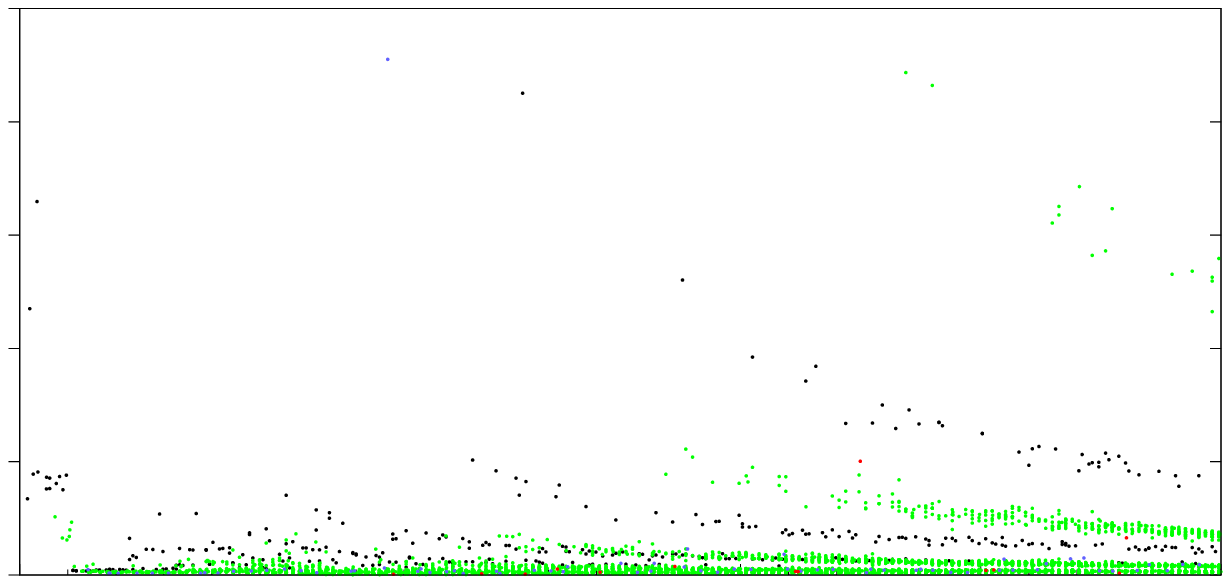}}%
    \gplfronttext
  \end{picture}%
\endgroup

%% file: T-small-fields.tex
\begingroup
  \makeatletter
  \providecommand\color[2][]{%
    \GenericError{(gnuplot) \space\space\space\@spaces}{%
      Package color not loaded in conjunction with
      terminal option `colourtext'%
    }{See the gnuplot documentation for explanation.%
    }{Either use 'blacktext' in gnuplot or load the package
      color.sty in LaTeX.}%
    \renewcommand\color[2][]{}%
  }%
  \providecommand\includegraphics[2][]{%
    \GenericError{(gnuplot) \space\space\space\@spaces}{%
      Package graphicx or graphics not loaded%
    }{See the gnuplot documentation for explanation.%
    }{The gnuplot epslatex terminal needs graphicx.sty or graphics.sty.}%
    \renewcommand\includegraphics[2][]{}%
  }%
  \providecommand\rotatebox[2]{#2}%
  \@ifundefined{ifGPcolor}{%
    \newif\ifGPcolor
    \GPcolorfalse
  }{}%
  \@ifundefined{ifGPblacktext}{%
    \newif\ifGPblacktext
    \GPblacktexttrue
  }{}%
  \let\gplgaddtomacro\g@addto@macro
  \gdef\gplbacktext{}%
  \gdef\gplfronttext{}%
  \makeatother
  \ifGPblacktext
    \def\colorrgb#1{}%
    \def\colorgray#1{}%
  \else
    \ifGPcolor
      \def\colorrgb#1{\color[rgb]{#1}}%
      \def\colorgray#1{\color[gray]{#1}}%
      \expandafter\def\csname LTw\endcsname{\color{white}}%
      \expandafter\def\csname LTb\endcsname{\color{black}}%
      \expandafter\def\csname LTa\endcsname{\color{black}}%
      \expandafter\def\csname LT0\endcsname{\color[rgb]{1,0,0}}%
      \expandafter\def\csname LT1\endcsname{\color[rgb]{0,1,0}}%
      \expandafter\def\csname LT2\endcsname{\color[rgb]{0,0,1}}%
      \expandafter\def\csname LT3\endcsname{\color[rgb]{1,0,1}}%
      \expandafter\def\csname LT4\endcsname{\color[rgb]{0,1,1}}%
      \expandafter\def\csname LT5\endcsname{\color[rgb]{1,1,0}}%
      \expandafter\def\csname LT6\endcsname{\color[rgb]{0,0,0}}%
      \expandafter\def\csname LT7\endcsname{\color[rgb]{1,0.3,0}}%
      \expandafter\def\csname LT8\endcsname{\color[rgb]{0.5,0.5,0.5}}%
    \else
      \def\colorrgb#1{\color{black}}%
      \def\colorgray#1{\color[gray]{#1}}%
      \expandafter\def\csname LTw\endcsname{\color{white}}%
      \expandafter\def\csname LTb\endcsname{\color{black}}%
      \expandafter\def\csname LTa\endcsname{\color{black}}%
      \expandafter\def\csname LT0\endcsname{\color{black}}%
      \expandafter\def\csname LT1\endcsname{\color{black}}%
      \expandafter\def\csname LT2\endcsname{\color{black}}%
      \expandafter\def\csname LT3\endcsname{\color{black}}%
      \expandafter\def\csname LT4\endcsname{\color{black}}%
      \expandafter\def\csname LT5\endcsname{\color{black}}%
      \expandafter\def\csname LT6\endcsname{\color{black}}%
      \expandafter\def\csname LT7\endcsname{\color{black}}%
      \expandafter\def\csname LT8\endcsname{\color{black}}%
    \fi
  \fi
    \setlength{\unitlength}{0.0500bp}%
    \ifx\gptboxheight\undefined%
      \newlength{\gptboxheight}%
      \newlength{\gptboxwidth}%
      \newsavebox{\gptboxtext}%
    \fi%
    \setlength{\fboxrule}{0.5pt}%
    \setlength{\fboxsep}{1pt}%
\begin{picture}(8502.00,3968.00)%
    \gplgaddtomacro\gplbacktext{%
      \csname LTb\endcsname%
      \put(462,440){\makebox(0,0)[r]{\strut{}$10$}}%
      \put(462,906){\makebox(0,0)[r]{\strut{}$11$}}%
      \put(462,1372){\makebox(0,0)[r]{\strut{}$12$}}%
      \put(462,1838){\makebox(0,0)[r]{\strut{}$13$}}%
      \put(462,2305){\makebox(0,0)[r]{\strut{}$14$}}%
      \put(462,2771){\makebox(0,0)[r]{\strut{}$15$}}%
      \put(462,3237){\makebox(0,0)[r]{\strut{}$16$}}%
      \put(462,3703){\makebox(0,0)[r]{\strut{}$17$}}%
      \put(657,220){\makebox(0,0){\strut{}$0$}}%
      \put(2387,220){\makebox(0,0){\strut{}$4000$}}%
      \put(4117,220){\makebox(0,0){\strut{}$8000$}}%
      \put(5847,220){\makebox(0,0){\strut{}$12000$}}%
      \put(7577,220){\makebox(0,0){\strut{}$16000$}}%
      \put(7709,440){\makebox(0,0)[l]{\strut{}$10$}}%
      \put(7709,906){\makebox(0,0)[l]{\strut{}$11$}}%
      \put(7709,1372){\makebox(0,0)[l]{\strut{}$12$}}%
      \put(7709,1838){\makebox(0,0)[l]{\strut{}$13$}}%
      \put(7709,2305){\makebox(0,0)[l]{\strut{}$14$}}%
      \put(7709,2771){\makebox(0,0)[l]{\strut{}$15$}}%
      \put(7709,3237){\makebox(0,0)[l]{\strut{}$16$}}%
      \put(7709,3703){\makebox(0,0)[l]{\strut{}$17$}}%
    }%
    \gplgaddtomacro\gplfronttext{%
      \csname LTb\endcsname%
      \put(5993,3530){\makebox(0,0)[l]{\strut{}quadratic}}%
      \csname LTb\endcsname%
      \put(5993,3310){\makebox(0,0)[l]{\strut{}biquadratic}}%
      \csname LTb\endcsname%
      \put(5993,3090){\makebox(0,0)[l]{\strut{}degree 6}}%
      \csname LTb\endcsname%
      \put(5993,2870){\makebox(0,0)[l]{\strut{}degree 21}}%
    }%
    \gplbacktext
    \put(0,0){\includegraphics{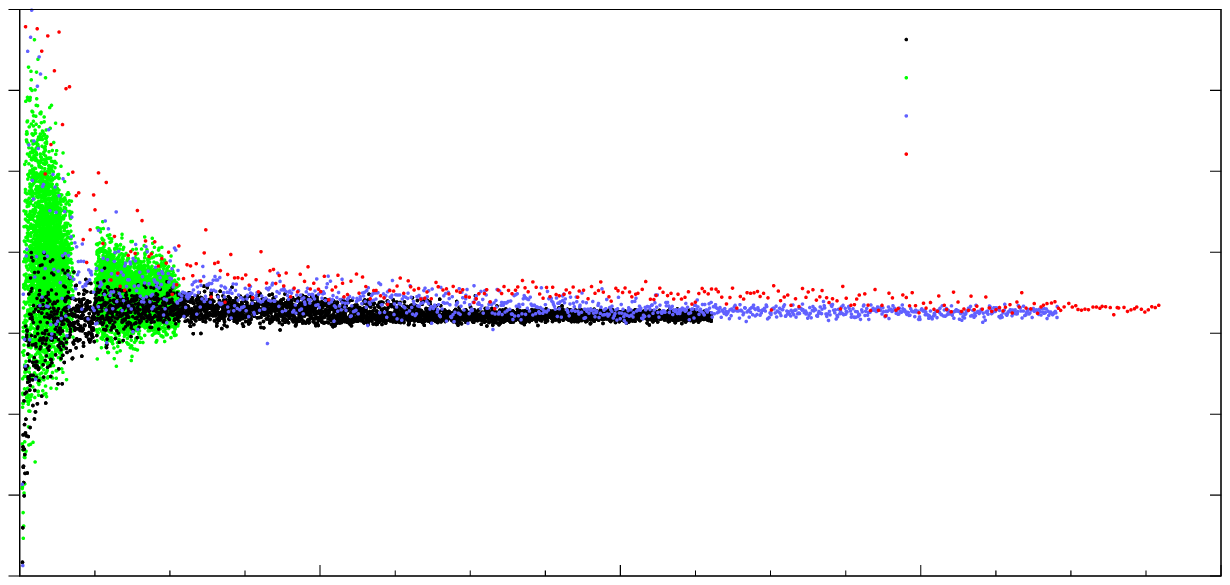}}%
    \gplfronttext
  \end{picture}%
\endgroup

%% file: T-small-fields-beppe.tex
\begingroup
  \makeatletter
  \providecommand\color[2][]{%
    \GenericError{(gnuplot) \space\space\space\@spaces}{%
      Package color not loaded in conjunction with
      terminal option `colourtext'%
    }{See the gnuplot documentation for explanation.%
    }{Either use 'blacktext' in gnuplot or load the package
      color.sty in LaTeX.}%
    \renewcommand\color[2][]{}%
  }%
  \providecommand\includegraphics[2][]{%
    \GenericError{(gnuplot) \space\space\space\@spaces}{%
      Package graphicx or graphics not loaded%
    }{See the gnuplot documentation for explanation.%
    }{The gnuplot epslatex terminal needs graphicx.sty or graphics.sty.}%
    \renewcommand\includegraphics[2][]{}%
  }%
  \providecommand\rotatebox[2]{#2}%
  \@ifundefined{ifGPcolor}{%
    \newif\ifGPcolor
    \GPcolorfalse
  }{}%
  \@ifundefined{ifGPblacktext}{%
    \newif\ifGPblacktext
    \GPblacktexttrue
  }{}%
  \let\gplgaddtomacro\g@addto@macro
  \gdef\gplbacktext{}%
  \gdef\gplfronttext{}%
  \makeatother
  \ifGPblacktext
    \def\colorrgb#1{}%
    \def\colorgray#1{}%
  \else
    \ifGPcolor
      \def\colorrgb#1{\color[rgb]{#1}}%
      \def\colorgray#1{\color[gray]{#1}}%
      \expandafter\def\csname LTw\endcsname{\color{white}}%
      \expandafter\def\csname LTb\endcsname{\color{black}}%
      \expandafter\def\csname LTa\endcsname{\color{black}}%
      \expandafter\def\csname LT0\endcsname{\color[rgb]{1,0,0}}%
      \expandafter\def\csname LT1\endcsname{\color[rgb]{0,1,0}}%
      \expandafter\def\csname LT2\endcsname{\color[rgb]{0,0,1}}%
      \expandafter\def\csname LT3\endcsname{\color[rgb]{1,0,1}}%
      \expandafter\def\csname LT4\endcsname{\color[rgb]{0,1,1}}%
      \expandafter\def\csname LT5\endcsname{\color[rgb]{1,1,0}}%
      \expandafter\def\csname LT6\endcsname{\color[rgb]{0,0,0}}%
      \expandafter\def\csname LT7\endcsname{\color[rgb]{1,0.3,0}}%
      \expandafter\def\csname LT8\endcsname{\color[rgb]{0.5,0.5,0.5}}%
    \else
      \def\colorrgb#1{\color{black}}%
      \def\colorgray#1{\color[gray]{#1}}%
      \expandafter\def\csname LTw\endcsname{\color{white}}%
      \expandafter\def\csname LTb\endcsname{\color{black}}%
      \expandafter\def\csname LTa\endcsname{\color{black}}%
      \expandafter\def\csname LT0\endcsname{\color{black}}%
      \expandafter\def\csname LT1\endcsname{\color{black}}%
      \expandafter\def\csname LT2\endcsname{\color{black}}%
      \expandafter\def\csname LT3\endcsname{\color{black}}%
      \expandafter\def\csname LT4\endcsname{\color{black}}%
      \expandafter\def\csname LT5\endcsname{\color{black}}%
      \expandafter\def\csname LT6\endcsname{\color{black}}%
      \expandafter\def\csname LT7\endcsname{\color{black}}%
      \expandafter\def\csname LT8\endcsname{\color{black}}%
    \fi
  \fi
    \setlength{\unitlength}{0.0500bp}%
    \ifx\gptboxheight\undefined%
      \newlength{\gptboxheight}%
      \newlength{\gptboxwidth}%
      \newsavebox{\gptboxtext}%
    \fi%
    \setlength{\fboxrule}{0.5pt}%
    \setlength{\fboxsep}{1pt}%
\begin{picture}(8502.00,3968.00)%
    \gplgaddtomacro\gplbacktext{%
      \csname LTb\endcsname%
      \put(594,440){\makebox(0,0)[r]{\strut{}$0$}}%
      \put(594,737){\makebox(0,0)[r]{\strut{}$0.1$}}%
      \put(594,1033){\makebox(0,0)[r]{\strut{}$0.2$}}%
      \put(594,1330){\makebox(0,0)[r]{\strut{}$0.3$}}%
      \put(594,1627){\makebox(0,0)[r]{\strut{}$0.4$}}%
      \put(594,1923){\makebox(0,0)[r]{\strut{}$0.5$}}%
      \put(594,2220){\makebox(0,0)[r]{\strut{}$0.6$}}%
      \put(594,2516){\makebox(0,0)[r]{\strut{}$0.7$}}%
      \put(594,2813){\makebox(0,0)[r]{\strut{}$0.8$}}%
      \put(594,3110){\makebox(0,0)[r]{\strut{}$0.9$}}%
      \put(594,3406){\makebox(0,0)[r]{\strut{}$1$}}%
      \put(594,3703){\makebox(0,0)[r]{\strut{}$1.1$}}%
      \put(789,220){\makebox(0,0){\strut{}$0$}}%
      \put(2453,220){\makebox(0,0){\strut{}$4000$}}%
      \put(4117,220){\makebox(0,0){\strut{}$8000$}}%
      \put(5781,220){\makebox(0,0){\strut{}$12000$}}%
      \put(7445,220){\makebox(0,0){\strut{}$16000$}}%
      \put(7577,440){\makebox(0,0)[l]{\strut{}$0$}}%
      \put(7577,737){\makebox(0,0)[l]{\strut{}$0.1$}}%
      \put(7577,1033){\makebox(0,0)[l]{\strut{}$0.2$}}%
      \put(7577,1330){\makebox(0,0)[l]{\strut{}$0.3$}}%
      \put(7577,1627){\makebox(0,0)[l]{\strut{}$0.4$}}%
      \put(7577,1923){\makebox(0,0)[l]{\strut{}$0.5$}}%
      \put(7577,2220){\makebox(0,0)[l]{\strut{}$0.6$}}%
      \put(7577,2516){\makebox(0,0)[l]{\strut{}$0.7$}}%
      \put(7577,2813){\makebox(0,0)[l]{\strut{}$0.8$}}%
      \put(7577,3110){\makebox(0,0)[l]{\strut{}$0.9$}}%
      \put(7577,3406){\makebox(0,0)[l]{\strut{}$1$}}%
      \put(7577,3703){\makebox(0,0)[l]{\strut{}$1.1$}}%
    }%
    \gplgaddtomacro\gplfronttext{%
      \csname LTb\endcsname%
      \put(5861,1273){\makebox(0,0)[l]{\strut{}quadratic}}%
      \csname LTb\endcsname%
      \put(5861,1053){\makebox(0,0)[l]{\strut{}biquadratic}}%
      \csname LTb\endcsname%
      \put(5861,833){\makebox(0,0)[l]{\strut{}degree 6}}%
      \csname LTb\endcsname%
      \put(5861,613){\makebox(0,0)[l]{\strut{}degree 21}}%
    }%
    \gplbacktext
    \put(0,0){\includegraphics{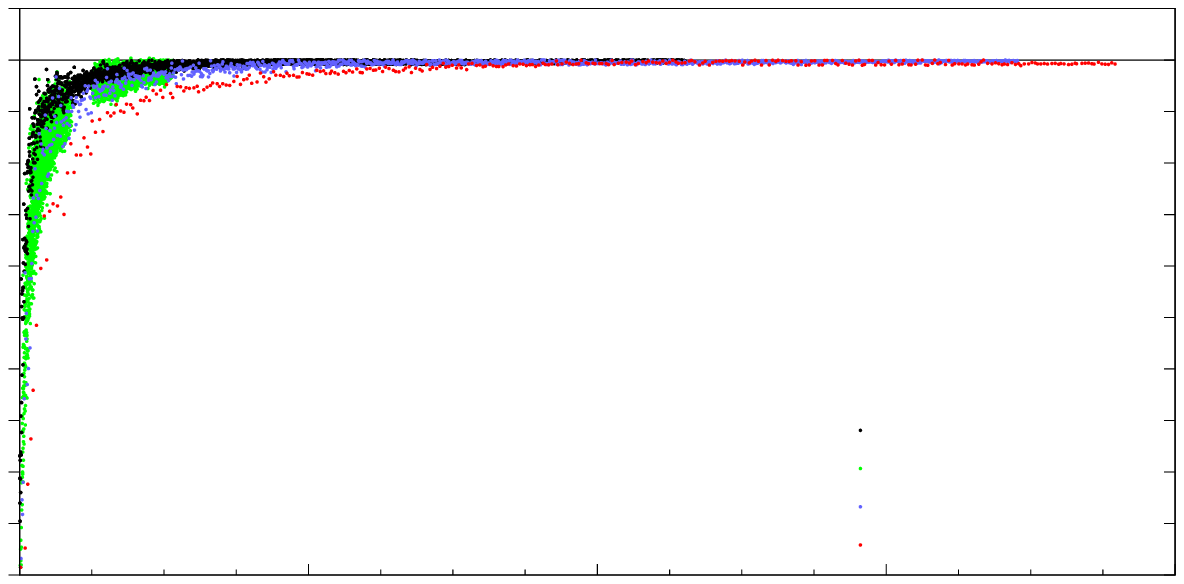}}%
    \gplfronttext
  \end{picture}%
\endgroup